\documentclass[reqno]{amsart}
\usepackage[utf8x]{inputenc}  

\usepackage{setspace}
\usepackage[left=3cm, right=3cm, bottom=4cm]{geometry}                 
\usepackage{enumerate}
\usepackage{graphicx} 
\usepackage{pgf,tikz}
\usetikzlibrary{arrows}
\usepackage{amssymb}
\usepackage{pdfsync}
\usepackage{mathrsfs}
\usepackage{hyperref} 
\usepackage{verbatim} 
\usepackage{epstopdf}
\usepackage{bbm} 
\usepackage[colorinlistoftodos,prependcaption,textsize=tiny]{todonotes}
\usepackage{xargs}
\usepackage{bm}

\usepackage{subcaption}

\usepackage[sort&compress,numbers]{natbib} 

\def\R{\mathbb{R}}

\def\N{\mathbb{N}}
\def\Z{\mathbb{Z}}
\def\C{\mathbb{C}}
\def\H{\mathbb{H}}

\def\Q{\mathcal{Q}}

\def\F{\mathcal{F}}
\def\V{\mathcal{V}}
\def\X{\mathcal{X}}

\def\supp{{\rm supp}}
\def\H{\mathcal{H}}

\renewcommand{\d}{\text{\rm d}}

\newcommand{\eps}{\varepsilon}

\newcommand{\mc}{\mathcal}

\newtheorem{theorem}{Theorem}
\newtheorem*{theoremA}{Theorem A}
\newtheorem*{theoremB}{Theorem B}
\newtheorem{corollary}[theorem]{Corollary}

\newtheorem*{definition*}{Definition}

\newtheorem{proposition}[theorem]{Proposition}
\newtheorem{lemma}[theorem]{Lemma}

\makeatletter
\DeclareFontFamily{U}{tipa}{}
\DeclareFontShape{U}{tipa}{m}{n}{<->tipa10}{}
\newcommand{\arc@char}{{\usefont{U}{tipa}{m}{n}\symbol{62}}}%

\makeatother

\numberwithin{equation}{section}

\allowdisplaybreaks

\newcommand{\intav}[1]{\mathchoice {\mathop{\vrule width 6pt height 3 pt depth  -2.5pt
\kern -8pt \intop}\nolimits_{\kern -6pt#1}} {\mathop{\vrule width
5pt height 3  pt depth -2.6pt \kern -6pt \intop}\nolimits_{#1}}
{\mathop{\vrule width 5pt height 3 pt depth -2.6pt \kern -6pt
\intop}\nolimits_{#1}} {\mathop{\vrule width 5pt height 3 pt depth
-2.6pt \kern -6pt \intop}\nolimits_{#1}}}

\newcommand{\intavl}[1]{\mathchoice {\mathop{\vrule width 6pt height 3 pt depth  -2.5pt
\kern -8pt \intop}\limits_{\kern -6pt#1}} {\mathop{\vrule width 5pt
height 3  pt depth -2.6pt \kern -6pt \intop}\nolimits_{#1}}
{\mathop{\vrule width 5pt height 3 pt depth -2.6pt \kern -6pt
\intop}\nolimits_{#1}} {\mathop{\vrule width 5pt height 3 pt depth
-2.6pt \kern -6pt \intop}\nolimits_{#1}}}

\title[Zeros of $L$-Functions in Low-Lying Intervals and de Branges spaces]{Zeros of $L$-Functions in Low-Lying Intervals \\ and de Branges spaces} 
\author[Ramos]{Antonio Pedro Ramos}
\subjclass[2020]{42A05, 46E22, 11F66, 11M26, 11M41}
\keywords{Families of $L$-functions; low-lying zeros; reproducing kernels; Hilbert spaces; de Branges spaces}

\address{SISSA - Scuola Internazionale Superiore di Studi Avanzati, Via Bonomea 265, 34136 Trieste, Italy}
\email{antonio.ramos@sissa.it}

\begin{document}

\begin{abstract} 
We consider a variant of a problem first introduced by Hughes and Rudnick (2003) and generalized by Bernard (2015) concerning conditional bounds for small first zeros in a family of $L$-functions. Here we seek to estimate the size of the smallest intervals centered at a low-lying height for which we can guarantee the existence of a zero in a family of $L$-functions. This leads us to consider an extremal problem in analysis which we address by applying the framework of de Branges spaces, introduced in this context by Carneiro, Chirre, and Milinovich (2022).

\end{abstract}

\maketitle 

\section{Introduction} 
This is a paper at the intersection of Analysis and Number Theory. We consider a question coming from the theory of $L$-functions, which we tackle by addressing a related extremal problem in analysis. These are the contents of Theorems \ref{MainThm} and \ref{EPSolution}. Theorem \ref{EPSolution} is in fact a specialization of the more general Theorem \ref{EPThm}, the main result of this paper, which solves a family of extremal problems in Hilbert spaces of entire functions.

\subsection{Zeros of \emph{L}-functions} As a consequence of a one-level density result, Hughes and Rudnick \cite{HR} conditionally proved the existence of small first zeros within the family of Dirichlet $L$-functions for large values of the conductor. First, they proved that for even Schwartz functions $\phi$ with $\supp \, \widehat{\phi} \subseteq [-2, 2]$, we have the \emph{one-level density}\footnote{In \cite{HR}, they use the term \emph{linear statistic} to mean the same thing. We adopt the terminology in \cite{ILS,KS1,KS2}.}
\begin{equation}\label{OneLevel}
    \lim_{q\rightarrow \infty}\frac{1}{q-2}\sum_{\chi \neq \chi_0} \sum_{\gamma_\chi} \phi\left( \frac{\gamma_\chi \log q}{2 \pi}\right) = \int_\R \phi(x) \, \d x,
\end{equation}
where the outside is a sum over non-principal Dirichlet characters modulo a prime $q$ and the inside sum is over their respective non-trivial zeros $\frac{1}{2} + i\gamma_\chi$. They then observed that, under the generalized Riemann Hypothesis (GRH), if the one-level density (\ref{OneLevel}) holds for even Schwartz functions $\phi$ with $\supp \, \widehat{\phi} \subseteq [-\Delta, \Delta]$, one can always obtain a bound for the height of the first zero in the family in terms of $\Delta > 0$. Explicitly, 
\begin{equation}\label{HRBound}
    \limsup_{q\rightarrow \infty} \, \min_{\substack{{\gamma_\chi}\\{\chi \neq \chi_0}}} \left|  \frac{\gamma_\chi \log q}{2 \pi} \right| \leq \frac{1}{2\Delta},
\end{equation}
which together with the previous observation means that, for large $q$, there exist zeros in this family within $\frac{1}{4}$ times the average spacing.

\smallskip

Bernard \cite{B}, and later Carneiro, Chirre, and Milinovich  \cite{CChiM}, extended the analysis of Hughes and Rudnick to more general families of $L$-functions. The context is the following: we consider families of automorphic objects $\F$.  Each $f \in \F$ has an associated $L$-function
\begin{equation*}
    L(s,f) = \sum_{n = 1}^\infty \lambda_f(n) n^{-s}
\end{equation*}
that can be continued analytically to an entire function, and its corresponding completed $L$-function $\Lambda(s,f) = L_\infty(s,f)L(s,f)$ satisfies a functional equation of the type
\begin{equation*}
    \Lambda(s,f) = \varepsilon_f \Lambda(1-s,\bar f),
\end{equation*}
where $|\varepsilon_f| =  1$ and $L(s,\bar f)$ is the dual $L$-function with Dirichlet series coefficients $\lambda_{\bar f}(n)= \overline{\lambda_f(n)}$. 
We denote the analytic conductor of $L(s,f)$ by $c_f$ and write its non-trivial zeros as $\rho_f = \frac{1}{2} + i \gamma_f$. Throughout, we assume GRH for these $L$-functions, meaning $\gamma_f \in \R$. We further assume that our family satisfies the assumptions
\begin{equation}\label{SelfDual}
    \begin{cases}
        f \in \mc F \implies \bar{f} \in \mc F; \\
        L_{\infty}(s,f) \neq 0 \text{ for all } \mathrm{Re}\, (s) = \frac{1}{2},
    \end{cases}
\end{equation}
which in in particular guarantee that if $\frac{1}{2} + i \gamma$ is a non-trivial zero within this family of $L$-functions, then so is its conjugate $\frac{1}{2} - i \gamma$.

As conjectured by Katz and Sarnak \cite{KS1, KS2}, for each natural family $\{ L(s,f) : f \in \mathcal F \}$ of $L$-functions there is an associated symmetry group $G = G(\mathcal F)$ which governs the distribution of its low-lying zeros. The group $G$ can be either unitary U, symplectic Sp, orthogonal O, even orthogonal SO(even), or odd orthogonal SO(odd). The idea is to consider averages over finite subsets of $\F$ ordered by the conductor
\begin{equation*} 
    \big\{f \in \ F: c_f = \Q \big\}
\end{equation*}
or, depending on the context,
\begin{equation*} 
    \big\{ f \in \ F: c_f \leq \Q \big\}
\end{equation*}
as $\Q \to \infty$. In both cases we denote the set under consideration by $\F(\Q)$ in order to unify the presentation. Now, for each of these symmetry groups, there is a corresponding density $W_G$ for which Katz and Sarnak conjectured the one level-density 
\begin{equation}\label{KSconj}
    \lim_{\Q\rightarrow \infty}\frac{1}{|\mathcal F(\Q)|}\sum_{f \in \mathcal F(\Q)} \sum_{\gamma_f} \phi\left( \frac{\gamma_f \log c_f}{2 \pi}\right) = \int_\R \phi(x) \, W_{G} (x) \, \d x
\end{equation}
for even Schwartz functions $\phi$ with compactly supported Fourier transform. Naturally, the zeros are counted with their multiplicities in the sums above.
These five densities are
\begin{align}
   \nonumber W_{\text{U}}(x) &= 1;\\
     \nonumber W_{\text{Sp}}(x) &= 1 - \frac{\sin 2 \pi x}{2 \pi x};\\
    \label{densities}W_{\text{O}}(x) &= 1 +\frac{1}{2} {\pmb \delta}_0(x);\\
    \nonumber W_{\text{SO(even)}}(x) &= 1 + \frac{\sin 2 \pi x}{2 \pi x};\\
    \nonumber W_{\text{SO(odd)}} &= 1 - \frac{\sin 2 \pi x}{2 \pi x}(x) + {\pmb  \delta}_0(x),
\end{align}
where $\pmb \delta_0$ is the Dirac delta at $x = 0$. Some examples of density results which  fall within this framework can be found in \cite{AlM, AnBal, BaiZ, BuFl, ChoKi, ConSna, DPRad, DueM1, DueM2, FI, GaZ, Gu, H-B, HM, HR, ILS, KS1, MPe, OSny, RiRo, Ro, Ru, ShSoT, Y}, with possible differences in setup and notation.
One-level densities for families of $L$-functions have been used to estimate the proportion of non-vanishing at the central point $s = \frac{1}{2}$ \cite{AnBal, BuFl, ConSna, DPRad, Fr, FrM, GaZ,HM, ILS, KS1, OSny}, the average rank of elliptic curves \cite{BaiZ, H-B,Y}, and the height of small first zeros \cite{B,CChiM, HR}. It is worth noting that the authors of \cite{CChiM} considered a generalization of the problem of the proportion of non-vanishing, moving away from the central point $s = \frac{1}{2}$ to low-lying heights of the critical line, where height is measured in terms of the analytic conductor.

Regarding small first zeros, the strategy common to \cite{B, CChiM, HR} is that one-level density results lead to certain extremal problems in analysis, which  in turn provide estimates of the height of the first zero in a family of $L$-functions. Bernard \cite{B} tackled these problems via a careful analysis of an associated  Volterra differential equation, while Carneiro, Chirre, and Milinovich \cite{CChiM} reframed them as a corollary of a more general result within the theory of de Branges spaces of entire functions.  In this paper we take up the latter approach to provide estimates for first zeros in intervals at low-lying heights of the critical line. If we consider the normalized zeros $\tilde \rho_f = \frac{1}{2} + i \tilde \gamma_f$ with $ \tilde \gamma_f = \gamma_f \frac{\log c_f}{2\pi}$, we want to know, given an $\alpha \in \R$, the size of the smallest interval of the critical line centered at $\frac{1}{2} + i \alpha$  for which we can guarantee the existence of a zero $\tilde \rho_f$ in a given family of $L$-functions.

\subsection{Hilbert spaces of entire functions}
The upper-bound (\ref{HRBound}) can be cast as the solution of an extremal problem in Paley--Wiener space, a classical example of a Hilbert space of entire functions. This is generalized in the case of the other symmetry groups by the de Branges spaces. 

We recall that an entire function $F: \C \to \C$ is said to be of exponential type if 
\begin{equation*}
    \tau (F) := \limsup_{|z|\to \infty}\, |z|^{-1}\log |F(z)| < \infty,
\end{equation*}
and, in that case, we say $\tau (F)$ is the exponential type of $F$. The spaces $\H_{G, \pi \Delta}$ of entire functions of exponential type at most $\pi \Delta$ with norm
\begin{equation*}
    \| F \|^2_{\H_{G, \pi \Delta}} := \int_\R |F(x)|^2 \,  W_G(x) \, \d x < \infty
\end{equation*}
constitute a Hilbert space.  We point out that $\H_{\text{U},\pi \Delta}$ equipped with the norm $\| \cdot \|_{\H_{\text{U}, \pi \Delta}}$ is the classical Paley--Wiener space. As demonstrated in \cite[Section 3]{CChiM}, an application of a Fourier uncertainty principle shows that, as sets, the $\H_{G, \pi \Delta}$ are all equal. This is the set $\H_{\pi \Delta}$ of functions of exponential type at most $\pi \Delta$ such that their restriction to $\R$ is square-integrable. The content of the Paley--Wiener theorem is that $F \in \H_{\pi \Delta}$ if and only if $F \in C(\R)\cap L^2(\R)$ and $\mathrm{supp}\,\widehat F \subseteq [-\Delta/2, \Delta/2]$.

If one defines 
\begin{equation}\label{EPCChiM}
    \mathbb A( {G, \pi \Delta, 0}) := \inf_{0 \neq f \in \mathcal H_{\pi \Delta}} \frac{\int_ \R |x f(x)|^2 \, W_G(x)\, \d x}{\int_\R|f(x)|^2 \, W_G(x) \, \d x},
\end{equation}
then the authors of \cite{CChiM} showed how to obtain Theorem A below. The presence of the parameter $0$ will become clear soon.

\begin{theoremA}[cf. {\cite[Theorem 9]{CChiM}}]
    Let $\{L(s,f), \, f \in \F\}$ be a family of L-functions with an associated symmetry group $G \in\{{\rm U, Sp, O, SO(even), SO(odd) }\}$, and assume that GRH holds for $L$-functions in this family. In addition, suppose (\ref{KSconj}) holds for even Schwartz functions $\phi$ with $\supp  \,  \widehat \phi \subseteq (-\Delta, \Delta)$ for a fixed $\Delta > 0$. Then for the families with $G \in\{\rm U, Sp,  SO(even) \}$
    \begin{equation*}
        \limsup_{\Q \rightarrow \infty} \, \min_{\substack{{\gamma_f}\\{f \in \F(Q)}}} \left|  \frac{\gamma_f \log c_f}{2 \pi} \right| \leq \sqrt{\mathbb A( {G, \pi \Delta, 0})}.
    \end{equation*}
    Moreover, for the families with $G \in\{\rm O, SO(odd) \}$, we may exclude the zeros at the central point $Z(f)$. Thus, denoting $G^\sharp = {\rm U}$ if $G = {\rm O}$ and $G^\sharp = {\rm Sp}$ if $G = {\rm SO(odd)}$, we have
    \begin{equation*}
        \limsup_{\Q \rightarrow \infty} \, \min_{\substack{{\gamma_f \notin Z(f) }\\{f \in \F(Q)}}} \left|  \frac{\gamma_f \log c_f}{2 \pi} \right| \leq \sqrt{\mathbb A( {G^\sharp, \pi \Delta, 0})}.
    \end{equation*}  
\end{theoremA}

In all of these spaces the evaluation functionals are continuous. In other words, for all $w \in \C$ the linear map $F \mapsto F(w) $ is continuous. By applying the Riesz representation theorem, we know that for each $w \in \C$ there is an associated entire function $K_{G, \pi \Delta}(w, \cdot) \in \H_{G, \pi \Delta}$ that corresponds to the evaluation functional at $w$, meaning
\begin{equation*}
    F(w) = \langle  F,  K_{G, \pi \Delta(w, \cdot)}\rangle_{\H_{G, \pi \Delta}} = \int_\R  F(x) \overline{K_{G, \pi \Delta}(w, x)} \, \d x
\end{equation*}
for all $F \in \H_{G, \pi \Delta}$. Such a $ K_{G, \pi \Delta}(w, \cdot)$ is called a reproducing kernel.  We thus say these spaces have the \emph{reproducing kernel property}. Carneiro, Chirre, and Milinovich studied the spaces $\H_{G, \pi \Delta}$  and found explicit expressions for their reproducing kernels when $\Delta > 0$ for the groups $G \in \{ \rm{U, O} \}$ and when $0 < \Delta \leq 2$ for $G  \in \{ \rm{Sp, SO(even), SO(odd)}\}$ (see \cite[Theorems 3--7]{CChiM}). They also established that the solution to the extremal problem posed by (\ref{EPCChiM}) was given by

\begin{theoremB}[cf. {\cite[Theorem 9 and Theorem 14]{CChiM}}]
    Let $G \in \{ \rm U, Sp, O, SO(even), SO(odd) \}$, and let $K = K_{G, \pi\Delta}$ be the reproducing kernel of the Hilbert space $(\H _{\pi\Delta}$, $\langle\cdot, \cdot\rangle_{G, \pi\Delta})$.
    Let $\xi_1$ be the first positive zero of the function $x \mapsto \mathrm{Re}\, ((1-ix)K(i,x))$. 
    Then
    \begin{equation*}
       \mathbb A( {G, \pi \Delta, 0})= \inf_{0 \neq f \in \mathcal H_{\pi \Delta}} \frac{\int_ \R |x f(x)|^2\, W_G(x)\, \d x}{\int_\R|f(x)|^2\, W_G(x) \, \d x} = \xi_1^2.
    \end{equation*}
\end{theoremB}
Together, Theorem A and Theorem B provide an upper bound for the minimal height of the first zero in a family of $L$-functions once knows the relevant reproducing kernels explicitly. Here we address the question of existence of zeros in small intervals of the critical line at different heights. In analogy with Theorem A, we define the following sharp constant for a given real number $\alpha$:
\begin{equation}\label{EPshift}
    \mathbb A (G, \pi \Delta, \alpha) := \inf_{0 \neq f \in \mathcal H_{\pi \Delta}} \frac{\int_ \R |(x - \alpha) f(x)|^2\, W_G(x)\, \d x}{\int_\R|f(x)|^2\, W_G(x) \, \d x}.
\end{equation}
\smallskip
Given this, we establish the following result.

\begin{theorem}\label{MainThm}
    Fix $\alpha \neq 0$ real and let $\{L(s,f), \, f \in \F\}$ be a family of L-functions with an associated symmetry group $G \in\{\rm U, Sp, O, SO(even), SO(odd) \}$, and assume that GRH holds for $L$-functions in this family. In addition, suppose (\ref{KSconj}) holds for even Schwartz functions $\phi$ with $\supp  \,  \widehat \phi \subseteq (-\Delta, \Delta)$ for a fixed $\Delta > 0$. Then 
    \begin{equation*}
        \limsup_{\Q \rightarrow \infty} \, \min_{\substack{{\gamma_f}\\{f \in \F(Q)}}} \left|  \frac{\gamma_f \log c_f}{2 \pi} - \alpha \right| \leq \sqrt{\mathbb A (G, \pi \Delta, \alpha)}.
    \end{equation*} 
\end{theorem}

\noindent {\sc Remark:}  In orthogonal and odd orthogonal families, where, due to the sign of the functional equation, we may expect at least part of the family to trivially vanish at the central point, we can further bound
\begin{equation*}
    \limsup_{\Q \rightarrow \infty} \, \min_{\substack{{\gamma_f}\\{f \in \F(Q)}}} \left|  \frac{\gamma_f \log c_f}{2 \pi} - \alpha \right| \leq \min \left\{  \sqrt{\mathbb A (G, \pi \Delta, \alpha)}, |\alpha|\right\}.
\end{equation*} 
\smallskip

The main difference with the problem considered previously is that when $\alpha \neq 0$ the symmetry is in a sense broken. All the measures considered originally by Katz and Sarnak are symmetric around the origin, as is the ``power weight'' $|x|^{2}$. Since this no longer holds for $|x - \alpha|^{2}$, we must go through a different path to provide an estimate for the sharp constants $\mathbb A (G, \pi \Delta, \alpha)$. We are inspired by the methods of \cite{CMCP} to prove the analogous of Theorem B below. 

\begin{theorem}\label{EPSolution}
    Fix $\alpha$ real. Let $G \in \{\rm U, Sp, O, SO(even), SO(odd) \}$, and let $K = K_{G, \pi\Delta}$ be the reproducing kernel of the Hilbert space $(\H _{\pi\Delta}$, $\langle\cdot, \cdot\rangle_{G,\pi\Delta})$   
    Let $\xi_1$ be the first positive zero of the function $x \mapsto K(\alpha +x , \alpha - x)$. Then
    \begin{equation*}
        \mathbb A (G, \pi \Delta, \alpha) = \inf_{0 \neq f \in \mathcal H_{\pi \Delta}} \frac{\int_ \R |(x - \alpha) f(x)|^2 \, W_G(x)\, \d x}{\int_\R|f(x)|^2 \, W_G(x) \, \d x} = \xi_1^2.
    \end{equation*}
\end{theorem}

\noindent {\sc Remark.} The above solution for the case $\alpha = 0$ is, of course, the same as that given by Theorem B. The first positive zero of $x \mapsto K(x , -x)$ and that of $x \mapsto \mathrm{Re}\,((1-ix)K(i,x))$ are both equal to the first positive zero of the function $x \mapsto A(x)$, an auxiliary function of the associated de Branges space, which will be defined in the next section.

\smallskip

\noindent {\sc Remark.} Notice from the above formula that the case of unitary symmetry $G = \rm U$ has the same solution for all $\alpha$ since $K(\alpha+x, \alpha -x) = \frac{\sin 2 \pi \Delta  x}{2 \pi x}$ (from this we can also see that $\mathbb A( {\mathrm U, \pi \Delta, \alpha}) \equiv \frac{1}{4 \Delta^2})$. One can reach the same conclusion by observing that the density $W_{\rm U}(x) \equiv 1$ is translation invariant.

\smallskip

We illustrate the bounds provided by the above theorems in Figure 1, where, for the five symmetry groups and $\Delta \in \{1, \frac{4}{3}, \frac{3}{2}, 2\}$, we graph the functions $\alpha \mapsto \sqrt{\mathbb A(G, \pi \Delta, \alpha)}$ using the explicit reproducing kernels computed in \cite{CChiM}. Our choice of values of $\Delta$ represented in the figure is inspired by one-level density results present in the literature. For instance, the  original result of Hughes and Rudnick for Dirichlet $L$-functions was proven for $\Delta = 2$ \cite[Theorem 3.1]{HR}. In \cite{ILS}, Iwaniec, Luo, and Sarnak establish such results in \cite[Theorems 1.1 and 1.3]{ILS} for certain  orthogonal families when $\Delta = 2$, and in \cite[Theorem 1.5]{ILS} for a particular symplectic family when $\Delta = \frac{3}{2}$, among other things. As a last example, Fouvry and Iwaniec \cite{FI} considered the family of $L$-functions $L(s, \Psi)$ where $\Psi$ runs over the characters of the ideal class group $\mc{C}\ell(K)$ of the imaginary quadratic field $K = \mathbb Q(\sqrt{-D})$, with $D > 3$ a square-free number congruent to $3$ modulo $4$. They showed (\ref{KSconj}) holds for this family when $\Delta = 1$ with $G = \rm Sp$ in their \cite[Theorem 1.1]{FI} and, by taking a certain average over $D$, extended it to $\Delta = \frac{4}{3}$ in \cite[Theorem 1.2]{FI}. 

Finally, a qualitative behavior that can be inferred from the graphs of Figure 1 is that as $\alpha$ grows large the values $\mathbb A (G, \pi \Delta, \alpha)$ get closer and closer to the constant value $\mathbb A( {\mathrm U, \pi \Delta, 0}) = \frac{1}{4 \Delta^2}$. Our next result shows that this is indeed the case.

\begin{theorem}\label{LimitThm}
    For all $G \in\{\rm U, Sp, O, SO(even), SO(odd) \}$ we have
    \begin{equation*}
        \lim_{|\alpha| \to \infty} \mathbb A(G,\pi \Delta, \alpha) = \mathbb A( {\mathrm U, \pi \Delta, 0}) = \frac{1}{4 \Delta^2}.
    \end{equation*}
\end{theorem}
We prove this in Section \ref{LimitProof}. To show this limit holds we work from the expression of the densities (\ref{densities}) and, through this proof, illustrate the relevance of Fourier uncertainty principles and of the reproducing kernel structure of the spaces under investigation.   

\begin{figure}[t]
    \begin{subfigure}{0.4\textwidth}
        \includegraphics[width=1\linewidth, height=4.5cm]{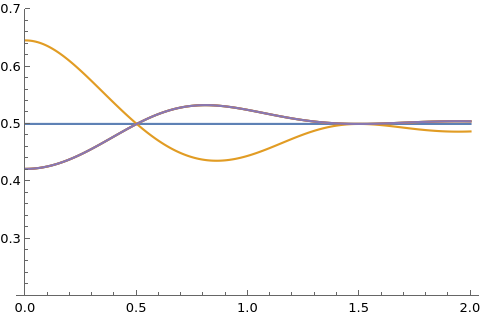} 
        \caption{$\Delta = 1.$}
    \end{subfigure}
    \begin{subfigure}{0.4\textwidth}
        \includegraphics[width=1\linewidth, height=4.5cm]{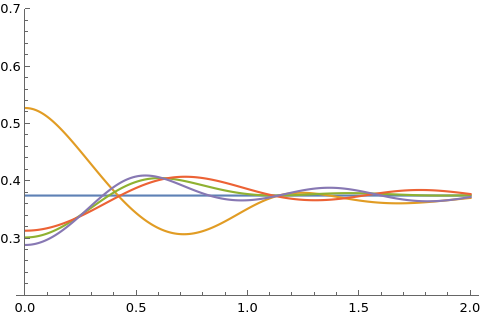}
     \caption{$\Delta = \frac{4}{3}$.}        
    \end{subfigure}
    \begin{subfigure}{0.1\textwidth}
        \includegraphics[width=1\linewidth, height=4.5cm]{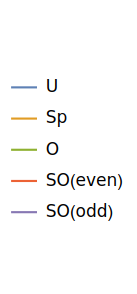}
    \end{subfigure}
      \begin{subfigure}{0.4\textwidth}
        \includegraphics[width=1\linewidth, height=4.5cm]{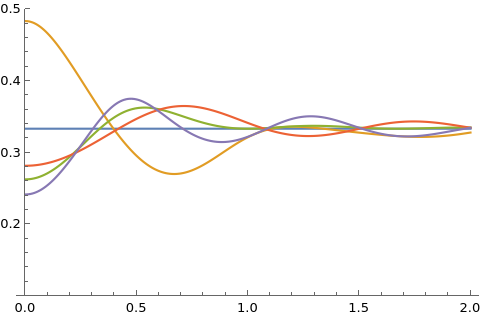} 
        \caption{$\Delta = \frac{3}{2}.$}
    \end{subfigure}
    \begin{subfigure}{0.4\textwidth}
        \includegraphics[width=1\linewidth, height=4.5cm]{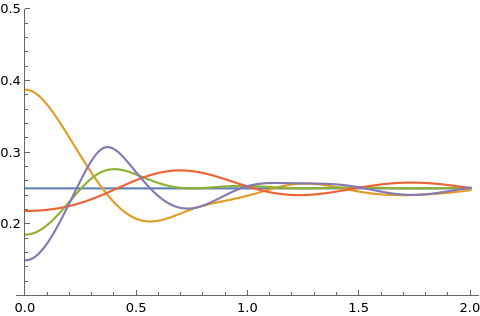}
     \caption{$\Delta = 2$.}        
    \end{subfigure}
    \begin{subfigure}{0.1\textwidth}
        \includegraphics[width=1\linewidth, height=4.5cm]{new-plots/newcaption.png}
    \end{subfigure}
    \caption{Values of $\sqrt{\mathbb A(G, \pi \Delta, \alpha)}$ as $\alpha$ varies, for different symmetry groups and different values of $\Delta$. Note $K_{\rm O, \pi \Delta} = K_{\rm SO(even), \pi \Delta} = K_{\rm SO(odd), \pi \Delta}$ when $\Delta = 1$.}
\end{figure}

\subsection{An extremal problem in de Branges spaces}

We connect the problem of determining the sharp constants $\mathbb A (G, \pi \Delta, \alpha)$ to the theory of de Branges spaces. Theorem \ref{EPSolution} is a consequence of the more general Theorem \ref{EPThm}, which is about sharp inequalities related to the operators of multiplication by $(z - \alpha)^{k}$ in a de Branges space, with $k \in \N$ and $\alpha \in \R$. This theorem addresses an extremal problem which is a generalization of the one studied by Carneiro et al in \cite[Section 8]{CMCP}, who take up the case $\alpha = 0$. In fact, this discussion parallels that of \cite{CMCP} in many respects. However, as mentioned previously, the broken symmetry introduces new challenges to consider.

We first recall some basic notions from this theory. See de Branges's monograph \cite{dB} for a more detailed exposition. We use the notation $f^*(z) = \overline{f(\bar z)}$ throughout.  A function that is analytic in the upper half-plane $\C^+$ is said to be of \emph{bounded type} if it can be expressed as the quotient of two bounded analytic functions in $\C^+$. For such a function $f$, one has \cite[Theorems 9 and 10]{dB} 
$$v(f) := \limsup_{y \rightarrow +\infty}\, {y}^{-1}\,{\log |f(iy)|} < \infty,$$
and we call $v(f)$ the mean type of $f$.

Let $E: \C \rightarrow \C$ be a \emph{Hermite-Biehler} function, that is, an entire function satisfying the inequality
\begin{equation*}
    |E^*(z)| < |E(z)|
\end{equation*}
for all $z \in \C^+$. The de Branges space $\H (E)$ associated to $E$ is the space of entire functions $f: \C \rightarrow \C$ such that
\begin{equation*}
    \|f\|_{\H (E)}^2 :=  \int_{-\infty}^\infty \left| \frac{f(x )}{E(x)}\right|^2 \d x < \infty,
\end{equation*}
and such that $f/E$ and $f^*/E$ are of bounded type with non-positive mean type. This space forms a reproducing kernel Hilbert space with the inner product
\begin{equation*}
    \langle f,g \rangle_{\H (E)} := \int_{-\infty}^\infty  \frac{f(x ) \overline{g(x)}}{|E(x)|^2} \, \d x.
\end{equation*}
To every such $E$ there exist unique real entire functions\footnote{We say an entire function $F: \C \to \C$ is real entire if its restriction to $\R$ is real-valued. Equivalently, $F(z) = F^*(z)$.}  $A: \C \rightarrow \C$ and $B:\C \rightarrow \C$ such that $E(z) = A(z) - iB(z)$. These companion functions must be 
\begin{equation*}
    A(z) = \frac{1}{2} \left(E(z) + E^*(z) \right) \text{  and  }  B(z) = \frac{i}{2} \left(E(z) - E^*(z) \right),
\end{equation*}
and we can write the reproducing kernels for these spaces in terms of these functions. For every $w \in \C$, the corresponding reproducing kernel is given by (see \cite[Theorem 19]{dB})
\begin{equation}\label{RepKernel}
    K(w,z) = \frac{B(z)A(\overline w) - A(z)B(\overline w)}{\pi (z -\overline w) },
\end{equation}
and, when $w = \bar z$,
\begin{equation*}
    K(\bar z,z) = \frac{1}{\pi}\left( B'(z) A(z) - A'(z)B(z) \right).
\end{equation*}
A classical example is the Paley-Wiener space $\H_{\rm U, \pi \Delta}$, which is the space $\H(E)$ when $E(z) = e^{-i \pi \Delta z}$. In this case the auxiliary functions are $A(z) = \cos \pi \Delta z$ and $B(z) = \sin \pi \Delta z$. 

We will be considering the following extremal problem, a generalization of the one defined in (\ref{EPshift}). 

\smallskip
\noindent \emph{Extremal Problem}: Fix $k \in \N$ and let $\X_k := \{ f \in \H (E) : z^k f \in \H (E) \}$. Find for $\alpha \in \R$
\begin{equation}\label{EP}
    (\mathbb{EP})(E,k,\alpha) := \inf_{0 \neq f \in \X_k} \frac{\|(z - \alpha)^k f\|_{\H (E)}^2}{\|f\|_{\H (E)}^2}.
\end{equation}

The fact that we can determine the solutions of equations of the form $(z - \alpha)^k = \eta$ for $\eta > 0$ allows us to adapt the argument of \cite{CMCP} to pin down a precise a value of $(\mathbb{EP})(E,k,\alpha)$ when the function $E$ satisfies the following hypotheses:
\begin{enumerate}[\hspace{1cm}(C1)]
        \item $E$ has no real zeros;
        \item the function $z \mapsto E(iz)$ is real entire;
        \item $A \notin \H (E)$.
\end{enumerate}
 Condition (C2) is equivalent to the fact that the auxiliary functions $A$ and $B$ are even and odd, respectively, while conditions (C1) and (C3) allow for the use of the following interpolation formula in de Branges spaces: 
 \begin{equation}\label{Int}
     f(z) = \sum_{A(\xi) = 0} \frac{f(\xi)}{K(\xi, \xi)}K(\xi, z)
 \end{equation}
 for all $f \in \H (E)$, with the sum above converging both uniformly in compact sets and in the norm of $\H (E)$. From (C1) and (C3), we also have  Parseval's identity
 \begin{equation}\label{Pars}
     \|f\|_{\H(E)}^2 = \sum_{A(\xi) = 0} \frac{|f(\xi)|^2}{K(\xi, \xi)}.
 \end{equation}
 For these results consult \cite[Theorem 22]{dB}.

 Let
 \begin{equation*}
     0 < \xi_1 < \xi_2 < \ldots
 \end{equation*}
 denote the sequence of positive zeros of $A$ and define the meromorphic function
 \begin{equation*}
     C(z) := \frac{B(z)}{A(z)}.
 \end{equation*}
 We can obtain the following result concerning the extremal problem defined in (\ref{EP}).

\begin{theorem}\label{EPThm}
    Suppose $E$ is a Hermite-Biehler function satisfying (C1), (C2), (C3) and  that $\X_k \neq \{0\}$. Setting $\lambda_0 := \left((\mathbb{EP})(E,k,\alpha)\right)^{1/2k}$, we have that $\lambda_0$ is the first positive zero of the equation
        \begin{equation}\label{RealSol}
            A(\alpha + \lambda)A(\alpha - \lambda) \det \V(\lambda) = 0,
        \end{equation}
        where $\V(\lambda)$ is the $k \times k$ matrix with entries
        \begin{equation*}
            \big( \V(\lambda) \big)_{lj} = \sum_{r = 0}^{2k - 1} \omega^{-r(l + j - 1)} C(\alpha + \omega^r\lambda) \qquad (1 \leq l,j \leq k)
        \end{equation*}
        and $\omega := e^{i \pi/ k}$.
\end{theorem}

We prove the above theorem in Section \ref{dBSec}. As in \cite{CMCP}, we convert the problem in $\H(E)$ to a problem in an appropriate space of sequences, which is Lemma \ref{SeqEquiv}. Then we discuss how Theorem \ref{EPSolution} may be seen as its consequence in Corollary \ref{CorTieBack}.

\section{Proof of Theorem 1: Existence of Zeros in Low-Lying Intervals}

Let us recall briefly the idea of Hughes and Rudnick to find the smallest interval of the central point for which there exists a zero in $\{L(s,f), \, f \in \F\}$. They construct an even Schwartz function $\phi$ with $\supp \, \widehat \phi  \subseteq (-\Delta, \Delta)$, which satisfies the properties
\begin{enumerate}[(i)]
    \item $\phi(x) \leq 0$ for $|x| \leq \beta$,
    \item $\phi(x) \geq 0$ for $|x| > \beta$
\end{enumerate}
for a given $\beta>0$ and apply the one-level density (\ref{KSconj}) to it. Such a $\phi$ is constructed by taking a non-negative $g$ Schwartz with $\supp \, \widehat g  \subseteq (-\Delta, \Delta)$ and setting $\phi(x) = (x^2 - \beta^2) \, g(x)$. 

The key insight here is that if we choose an appropriate $\beta>0$, we can find a $g$ for which
\begin{equation}\label{NegCond}
     \int_\R \phi(x)\, W_G(x)\, \d x = \int_\R (x^2 - \beta^2) \, g(x)\, W_G(x)\, \d x < 0
\end{equation}
and this implies by (\ref{KSconj}) and properties (i) and (ii) that for large values of $\Q$ the family $\{L(s,f), \, f \in \F(\Q)\}$ has a zero in the interval $[-\beta,\beta]$. Of course, it is in our interest to make the value of $\beta$ as small as possible, and this is where the extremal problem in an appropriate function space arises, for (\ref{NegCond}) is equivalent to
\begin{equation}\label{HRIneq}
    \frac{\int_ \R x^2 \, g(x)\, W_G(x)\, \d x}{\int_\R g(x) \, W_G(x) \, \d x} < \beta^2.
\end{equation}
 From here, we can get to the formulation contemplated in Theorem B. Indeed, Krein's decomposition theorem states that non-negative entire functions of exponential type at most $2 \pi \Delta$ can be decomposed into $g(x) = |f(x)|^2$, with $f$ entire of exponential type at most $\pi \Delta$ (see \cite[p. 154]{A}), so that inequality (\ref{HRIneq}) becomes
 \begin{equation*}
    \frac{\int_ \R x^2 \, |f(x)|^2 \, W_G(x)\, \d x}{\int_\R |f(x)|^2 \, W_G(x) \, \d x} < \beta^2
\end{equation*}
 and, in trying to minimize $\beta$, we are led to consider the extremal problem
\begin{equation*}
    \inf_{\substack{{0 \neq f \in \mathcal H_{\pi \Delta}}\\{f \in \mc S(\R)}}} \frac{\int_ \R |x f(x)|^2 \, W_G(x)\, \d x}{\int_\R|f(x)|^2\, W_G(x) \, \d x}.
\end{equation*}
By the density of functions of exponential type at most $\pi \Delta$ which are Schwartz on the real line, we can take the infimum above over the whole space $\H_{\pi \Delta}$ (see \cite[proof of Theorem 2]{CChiM}).

Now one would try to construct, for an arbitrary $\alpha \in \R$, an even Schwartz function $\phi_\alpha$ with  $\supp \, \widehat \phi_\alpha  \subseteq (-\Delta, \Delta)$ satisfying

\begin{enumerate}[(i')]
    \item $\phi_\alpha(x) \leq 0$ for $|x- \alpha| \leq \beta$;
    \item $\phi_\alpha(x) \geq 0$ for $|x - \alpha| > \beta$,
\end{enumerate}
in the same way. However, given the fact $\phi_\alpha$ must be even, the above properties cannot hold simultaneously when $\alpha \neq 0$. A more symmetric set of conditions would be

\begin{enumerate}[(I)]
    \item $\phi_\alpha(x) \leq 0$ for $x \in [\alpha - \beta, \alpha + \beta] \cup [-\alpha - \beta, -\alpha + \beta]$,
    \item $\phi_\alpha(x) \geq 0$ for $x \notin [\alpha - \beta, \alpha + \beta] \cup [-\alpha - \beta, -\alpha + \beta]$.
\end{enumerate}

Trying to mimic the previous construction for $\alpha \neq 0$ would give us the candidate $((x - \alpha)^2 - \beta^2)((x+ \alpha)^2 - \beta^2)\, g(x)$, for a non-negative even Schwartz function $g$ with $\supp \, \widehat g  \subseteq (-\Delta, \Delta)$. This leads to an overly complicated extremal problem. A simple observation allows us to fix this. On second inspection, one may notice that property (i), and its analogous (I), are in fact superfluous to conclude the existence of a zero. Because if (\ref{NegCond}) holds, only the fact that $\phi(x) \geq 0$ when $x \notin [\alpha - \beta, \alpha + \beta] \cup [-\alpha - \beta, -\alpha + \beta]$, imply there must be a zero in the family of $L$-functions in $[\alpha - \beta, \alpha + \beta] \cup [-\alpha - \beta, -\alpha + \beta]$ given that the integral is negative.

We may thus propose the following: define $\varphi_\alpha(x)= ((x -\alpha)^2-\beta^2))\, g(x)$, with $g$ non-negative, Schwartz and with $\supp \, \widehat g  \subseteq (-\Delta, \Delta)$. We do not suppose $g$ is even this time, but instead take our candidate to be the symmetrized version of $\varphi_\alpha$
\begin{equation*}
    \phi_\alpha(x) = \frac{ \varphi_\alpha(x) + \varphi_\alpha(-x)}{2}.
\end{equation*}
The function $\phi_\alpha$ is, of course, Schwartz with $\supp \, \widehat {\phi_\alpha} \subseteq (-\Delta, \Delta)$ and satisfies property (II). We can apply the one-level density, and if we obtained again that
\begin{equation}\label{NegCondA}
    \lim_{\Q\rightarrow \infty}\frac{1}{|\mathcal F(\Q)|}\sum_{f \in \mathcal F(\Q)} \sum_{\gamma_f} \phi_\alpha\left( \frac{\gamma_f \log c_f}{2 \pi}\right) = \int_\R \phi_\alpha(x) \, W_{G} (x) \, \d x < 0,
\end{equation}
it would imply together with the properties (\ref{SelfDual}) that
\begin{equation}\label{ZeroExistShift}
    \limsup_{\Q \rightarrow \infty} \, \min_{\substack{{\gamma_f}\\{f \in \F(Q)}}} \left|  \frac{\gamma_f \log c_f}{2 \pi} - \alpha \right| \leq \beta.
\end{equation}

Let us rewrite the right-hand side of (\ref{NegCondA}):
\begin{align*}
    0 &>  \int_\R \phi_\alpha(x) \, W_{G} (x) \, \d x =\int_\R \frac{ \varphi_\alpha(x) + \varphi_\alpha(-x)}{2} \, W_{G} (x) \, \d x  \\
    &= \frac{1}{2}\int_\R \varphi_\alpha(x) \, W_{G} (x) \, \d x + \frac{1}{2}\int_\R \varphi_\alpha(-x) \, W_{G} (x) \, \d x \\
    &= \frac{1}{2}\int_\R \varphi_\alpha(x) \, W_{G} (x) \, \d x + \frac{1}{2}\int_\R \varphi_\alpha(x) \, W_{G} (-x) \, \d x  \\
    &= \int_\R \varphi_\alpha(x) \, W_{G} (x) \, \d x \\
    &=\int_\R ((x -\alpha)^2-\beta^2)) \, g(x) \, W_{G} (x) \, \d x  ,
\end{align*}
by the evenness of the density $W_G(x)$. This, in turn, is equivalent to
\begin{equation*}
    \frac{\int_ \R (x - \alpha)^2 \, g(x) \, W_G(x)\, \d x}{\int_\R g(x) \, W_G(x) \, \d x} < \beta^2,
\end{equation*}
so that, again, by Krein's decomposition and the density of Schwartz functions, we obtain the relation to our desired extremal problem given by (\ref{EPshift}). Combining this observation with (\ref{ZeroExistShift}), we arrive at the content of Theorem \ref{MainThm}.

\section{Proof of Theorem \ref{LimitThm}: the Limit of \texorpdfstring{$\mathbb A (G, \pi \Delta, \alpha) $}{A(G, a)} as \texorpdfstring{$|\alpha| \to \infty$}{a to infinity}}\label{LimitProof}

    First notice that $f(z) \in \H_{\pi \Delta}$ if and only if $f(z-\alpha)\in  \H_{\pi \Delta}$, so by a simple change of variables  (\ref{EPshift}) can be rewritten to
    \begin{equation*}
        \mathbb A (G, \pi \Delta, \alpha) = \inf_{0 \neq f \in \mathcal H_{\pi \Delta}} \frac{\int_ \R |x f(x)|^2 \, W_G(x + \alpha)\, \d x}{\int_\R|f(x)|^2\, W_G(x + \alpha) \, \d x}.
    \end{equation*}
    Moreover, by the evenness of $W_G$, we may consider $\alpha > 0$ without loss of generality.
    We may also write a general formula for the densities (\ref{densities}) in the compact way
    \begin{equation*}
        W_G(x) = 1 + \gamma \, \frac{\sin 2 \pi x}{2 \pi x} + \eta \, {\pmb \delta}_0(x),
    \end{equation*}
    with $\gamma \in \{-1,0,1\}$ and $\eta \in \{0, 1/2, 1 \}$. 

    Fixing any non-zero $f\in \H_{\pi \Delta}$ such that $zf(z) \in \H_{\pi \Delta}$, we thus have
    for all $\alpha$
    \begin{align*}
          \mathbb A (G, \pi \Delta, \alpha) \leq \frac{\int_ \R |x f(x)|^2 \, \d x + \gamma \int_ \R |x f(x)|^2  \, \frac{\sin 2 \pi (x + \alpha)}{2 \pi (x + \alpha)}\, \d x + \eta \, |\alpha \, f(-\alpha)|^2}{
             \int_ \R | f(x)|^2 \, \d x + \gamma \int_ \R | f(x)|^2 \, \frac{\sin 2 \pi (x + \alpha)}{2 \pi (x + \alpha)}\, \d x + \eta \, |f(-\alpha)|^2}.
    \end{align*}
    By the Riemann-Lebesgue lemma, $x f(x) \to 0$ as $|x| \to \infty$ and, by the dominated convergence theorem,
    \begin{equation*}
        \lim_{\alpha \to \infty}\int_ \R |x f(x)|^2 \, \frac{\sin 2 \pi (x + \alpha)}{2 \pi (x + \alpha)}\, \d x = \lim_{\alpha \to \infty}\int_ \R | f(x)|^2 \, \frac{\sin 2 \pi (x + \alpha)}{2 \pi (x + \alpha)}\, \d x = 0.
    \end{equation*}
    Hence,
    \begin{equation}\label{limsupAlpha}
        \limsup_{\alpha \to \infty} \mathbb A (G, \pi \Delta, \alpha) \leq \frac{\int_ \R |x f(x)|^2 \,  \d x}{\int_\R|f(x)|^2\, \d x}
    \end{equation}
    for all non-zero $f \in \H_{\pi \Delta}$. 

    We note that 
    \begin{equation}\label{liminfNonZero}
        \liminf_{\alpha \to \infty} \mathbb A(G, \pi \Delta, \alpha)> 0,
    \end{equation}
    as a consequence of the equivalence of norms
    \begin{equation*}
        c_1 \| F \|^2_{\H_{G, \pi \Delta}} \leq \| F \|^2_{\H_{\rm{U}, \pi \Delta}} \leq c_2 \| F \|^2_{\H_{G, \pi \Delta}}
    \end{equation*}
    and the fact that $\mathbb A (\mathrm U, \pi \Delta, \alpha) \equiv \frac{1}{4\Delta^2}$. For a sequence $\alpha_n \to \infty$ such that
    \begin{equation*}
         \liminf_{\alpha \to \infty} \mathbb A(G, \pi \Delta, \alpha) = \lim_{n \to \infty}  \mathbb A(G, \pi \Delta, \alpha_n),
    \end{equation*}
     we can associate a sequence of extremizers  for $\mathbb A(G, \pi \Delta, \alpha_n)$ denoted by $f_n \in \H_{\pi \Delta}$. We postpone the proof of the existence of such extremizers to the next section (see Proposition \ref{QualP}), but in fact this proof works just as well taking near-extremizers instead. We further  normalize these functions to have 
     \begin{equation}\label{NormalizationAlpha}
     \int_ \R |x f_{n}(x)|^2 \, \d x + \gamma \int_ \R |x f_{n}(x)|^2  \, \frac{\sin 2 \pi (x + \alpha_n)}{2 \pi (x + \alpha_n)}\, \d x + \eta \, |\alpha_n \,f_{n}(-\alpha_n)|^2 = 1.
     \end{equation}
     By the uncertainty principle of Nazarov\footnote{The uncertainty principle of Nazarov states that there exists a universal constant $C$ such that if $A,B \subseteq \R$ are sets of finite Lebesgue measure and $F \in L^2(\R)$, then $\int_\R |F(x)|^2 \, \d x \leq Ce^{C|A||B|} \left( \int_{\R \setminus A} |F(x)|^2 \, \d x + \int_{\R \setminus B} |\widehat F(y)|^2 \, \d y\right)$.} \cite{N}, $\{z f_{n}\}$ is a uniformly bounded sequence in the Hilbert space $\H_{\rm U, \pi \Delta}$, since it implies
     \begin{align}
         \nonumber \int_\R |x f_n (x)|^2 \d x &\ll  \int_{|x + \alpha_n| \geq 1} |x f_n (x)|^2 \d x \\
         \label{UnifBound}&\ll \int_{|x + \alpha_n| \geq 1} |x f_n (x)|^2  \bigg\{ 1 + \gamma \, \frac{\sin 2 \pi (x + \alpha_n)}{2 \pi (x+ \alpha_n)} \bigg\}  \d x \leq 1, 
     \end{align}
     where the implicit constants are uniform on $n$. Moreover, having in mind (\ref{liminfNonZero}) and
     \begin{equation*}
        \int_ \R | f_{n}(x)|^2  \, \d x + \gamma \int_ \R | f_{n}(x)|^2  \, \frac{\sin 2 \pi (x + \alpha_n)}{2 \pi (x + \alpha_n)}\, \d x + \eta \, |\,f_{n}(-\alpha_n)|^2 = \frac{1}{\mathbb A(G, \pi \Delta, \alpha_n)},
     \end{equation*}
     we obtain that  $\{f_{n}\}$ is uniformly bounded by the same reasoning. Thus, we have the weak convergences in $\H_{ \rm U, \pi \Delta}$
     \begin{equation*}
         \begin{cases}
             f_{n} \rightharpoonup f\\
            z f_{n} \rightharpoonup g
         \end{cases} \text{ as } n \to \infty,
     \end{equation*}
     which are also pointwise convergences because evaluation functionals are continuous in this space. In particular, it must be the case that $g(z) =z f(z)$.
     
     Fix $R > 0$ and $0 < r < \limsup_{\alpha \to \infty} \mathbb A(G, \pi \Delta, \alpha)^{-1}$. For $n$ big enough,
    \begin{align}
        \nonumber&r \leq \int_\R |f_{n}(x)|^2 \bigg\{  1 + \gamma \, \frac{\sin 2 \pi (x + \alpha_n)}{2 \pi (x+ \alpha_n)} + \eta \, {\pmb \delta}_0(x+\alpha_n) \bigg\} \d x \\
        \nonumber&\leq \int_{-R}^R|f_{n}(x)|^2 \bigg\{  1 + \gamma \, \frac{\sin 2 \pi (x + \alpha_n)}{2 \pi (x+ \alpha_n)} \bigg\}\d x +  \int_{|x| > R} |f_{n}(x)|^2 \bigg\{  1 + \gamma \, \frac{\sin 2 \pi (x + \alpha_n)}{2 \pi (x+ \alpha_n)} \bigg\}\d x +  |f_{n}(-\alpha_n)|^2 \\
        &\label{MassNotEsc}\leq \int_{-R}^R |f_{n}(x)|^2 \bigg\{  1 + \gamma \, \frac{\sin 2 \pi (x + \alpha_n)}{2 \pi (x+ \alpha_n)} \bigg\}\d x  + \frac{C}{R} + |f_n(-\alpha_n)|^2,
    \end{align}
    where from the second to the last line we used (\ref{UnifBound}). But  by the reproducing kernel property of $\H_{\rm U, \pi \Delta}$ and Cauchy-Schwarz,
    \begin{equation*}
        |x f_{n}(x)| = \left| \int_\R t f_{n}(t) \frac{\sin \pi \Delta (t-x)}{\pi(t-x)}\d t \right| \leq \left( \int_\R |t  f_{n}(t)|^2 \, \d t  \right)^{1/2} \left( \int_\R \left|\frac{\sin \pi \Delta (t-x)}{\pi(t-x)}\right|^2 \, \d t  \right)^{1/2} \ll_{\Delta} 1
    \end{equation*}
    for all $x \in \R$, and so, for $|x| \geq 1$,
    \begin{equation*}
        | f_{n}(x)| \leq \frac{1}{|x|}|x  f_{n}(x)|\ll \frac{1}{|x|},
    \end{equation*}
    meaning $| f_{n}(-\alpha_n)|^2 = O\left(\frac{1}{\alpha_n^2}\right)$ as $n \to \infty$. Sending $n \to \infty$ in (\ref{MassNotEsc}),
    \begin{equation*}
        r \leq  \int_{|x| \leq R}|f(x)|^2\, \d x  + \frac{C}{R}.
    \end{equation*}
    Now send $r \to \limsup_{\alpha \to \infty}\mathbb A(G, \pi \Delta, \alpha)^{-1}$ and $R \to \infty$ and rearrange to obtain
    \begin{equation*}
        \liminf_{\alpha \to \infty} \mathbb A (G, \pi \Delta, \alpha) \geq \frac{1}{\int_\R|f(x)|^2\, \d x}.
    \end{equation*}
    Thus, if we prove $\int_ \R |x f(x)|^2 \,  \d x \leq 1$, we are done, since we also have (\ref{limsupAlpha}). But indeed, by Fatou's lemma and (\ref{NormalizationAlpha}),
    \begin{equation*}
        \int_ \R |x f(x)|^2 \,  \d x \leq \liminf_{n \to \infty} \int_\R |x f_{n}(x)|^2 \bigg\{  1 + \gamma \, \frac{\sin 2 \pi (x + \alpha_n)}{2 \pi (x+ \alpha_n)} \bigg\}  \d x \leq 1.
    \end{equation*}

\section{The Extremal Problem in de Branges Spaces}\label{dBSec}

We need a few preliminary results concerning the Extremal Problem defined in (\ref{EP}) before embarking on the proof of Theorem \ref{EPThm}.

\begin{proposition}\label{QualP}
    Let $E$ be a Hermite-Biehler function and suppose $\X_k  = \{ f \in \H (E) : z^k f \in \H (E) \}\neq \{0\}$. The following statements hold:
    \begin{enumerate}[(i)]
        \item $(\mathbb{EP})(E,k,\alpha) > 0$.
        \item There exist extremizers for $(\mathbb{EP})(E,k,\alpha)$.
    \end{enumerate}
\end{proposition}

\begin{proof}
    Essentially as found in \cite{CMCP}, we include it here for the sake of completion. 

     First observe that $f \in \X_k$ if and only if $(z - \alpha)^k f \in \H(E)$. This follows because multiplication by $(z - \alpha)^k$ or $z^k$ does not change the mean type of a function and the integrability condition is in essence the same. Indeed, outside a compact set the function $|x - \alpha|^{2k}$ is of the order of $|x|^{2k}$, and vice versa.
    
    Note now that the function $x \mapsto K(x,x)|E(x)|^{-2}$ is continuous, and, in particular, locally integrable. We can thus find $0 < \eta < 1$ for which
    \begin{equation*}
        \int_{-\eta + \alpha}^{\eta+ \alpha}\frac{K(x,x)}{|E(x)|^{2}} \, \d x \leq \frac{1}{2}.
    \end{equation*}
    Thus, from the reproducing kernel property and Cauchy-Schwarz
    \begin{align*}
         \int_{-\eta + \alpha}^{\eta+ \alpha} \left|\frac{f(x )}{E(x)}\right|^2 \, \d x =  \int_{-\eta + \alpha}^{\eta+ \alpha} \left|\frac{\langle f, K(x, \cdot)\rangle_{\H(E)}}{E(x)}\right|^2 \, \d x \leq \| f\|_{\H(E)}^2 \int_{-\eta + \alpha}^{\eta+ \alpha}\frac{K(x,x)}{|E(x)|^{2}} \, \d x \leq \frac{\| f\|_{\H(E)}^2}{2}.
    \end{align*}
    So that 
    \begin{align*}
          \int_{\R} \left|\frac{(x - \alpha)^k f(x )}{E(x)}\right|^2 \, \d x \geq  \int_{|x - \alpha| \geq \eta} \left|\frac{(x - \alpha)^k f(x )}{E(x)}\right|^2 \, \d x  
          \geq \eta^{2k} \int_{|x - \alpha| \geq \eta} \left|\frac{ f(x )}{E(x)}\right|^2 \, \d x \geq \eta^{2k} \frac{\| f\|_{\H(E)}^2}{2},
    \end{align*}
    which implies $ (\mathbb{EP})(E,k,\alpha) \geq \frac{\eta^{2k}}{2}  > 0$.

    To prove (ii), take a minimizing sequence $\{ f_n\} \subseteq \H(E)$ normalized by $\|(z - \alpha)^k f_n\|_{\H (E)}^2  = 1$, implying
    \begin{equation*}
        \lim_{n \to \infty} \|f_n\|^2_{\H(E)} =  (\mathbb{EP})(E,k,\alpha)^{-1}.
    \end{equation*}
    This is a bounded sequence in $\H(E)$, so by the Banach--Alaoglu theorem it converges weakly to a $g \in \H(E)$. Since evaluation functionals are continuous, $f_n(z) \to g(z)$ for all $z \in \C$. So that applying Fatou's lemma
    \begin{equation*}
        \int_{\R}  \bigg|\frac{(x - \alpha)^k g(x )}{E(x)}\bigg|^2 \, \d x \leq \lim_{n \to \infty} \int_{\R}  \bigg|\frac{(x - \alpha)^k f_n(x )}{E(x)}\bigg|^2 \, \d x = 1,
    \end{equation*}
    verifying $g \in \X_k.$
    
    If $0 < r < (E,k,\alpha)^{-1}$, then for $n$ big enough and for $R > 0$,
    \begin{align*}
        r &\leq \int_{\R} \bigg|\frac{f_n(x )}{E(x)}\bigg|^2 \, \d x \\
        &= \int_{-R + \alpha}^{R + \alpha} \bigg|\frac{f_n(x )}{E(x)}\bigg|^2 \, \d x + \int_{|x - \alpha| \geq R}  \bigg|\frac{f_n(x )}{E(x)}\bigg|^2 \, \d x \\
        &\leq  \int_{-R + \alpha}^{R + \alpha} \bigg|\frac{f_n(x )}{E(x)}\bigg|^2 \, \d x + \frac{1}{R^{2k}} \int_{|x - \alpha| \geq R}  \bigg|\frac{(x - \alpha)^k f_n(x )}{E(x)}\bigg|^2 \, \d x\\
        &\leq \int_{-R + \alpha}^{R + \alpha} \bigg|\frac{f_n(x )}{E(x)}\bigg|^2 \, \d x + \frac{1}{R^{2k}},
    \end{align*}
    by the normalization $\|(z - \alpha)^k f_n\|^2_{\H(E)} = 1$. Sending $n \to \infty$ and applying the dominated convergence theorem,
    \begin{equation*}
        r \leq \int_{-R + \alpha}^{R + \alpha} \bigg|\frac{g(x )}{E(x)}\bigg|^2 \, \d x + \frac{1}{R^{2k}}.
    \end{equation*}
    Now we can send $R \to \infty$ and $r \to (E,k,\alpha)^{-1}$ to conclude $g$ is the desired extremizer.
\end{proof}
\noindent {\sc Remark.} The proof above is valid without assuming any of three conditions (C1)--(C3) on the Hermite-Biehler function $E$. Hypotheses (C2) and (C3) are clearly not being used, but even (C1) may be discarded, because possible real zeros of $|E(x)|^2$ are cancelled out by those of $K(x,x)$, so that $x \mapsto K(x,x)|E(x)|^{-2}$ is in any case continuous.

\smallskip

 We now establish an equivalence between the extremal problem in $\H(E)$ to a corresponding one in an appropriate space of sequences.

\begin{lemma}\label{SeqEquiv}
    Let $E$ be a Hermite-Biehler function satisfying (C1) and (C3) such that $\X_k \neq \{0\}$. Set $c_n := -A'(\xi_n)/B(\xi_n)$ and let $X$ be the space of real-valued sequences $\{a_n\}_{\Z^*}$ such that
    \begin{enumerate}[(i)]
        \item $\displaystyle \sum_{\substack{{n = -\infty}\\{n \neq 0}}}^\infty c_n a_n^2 \xi_n^{2k} < \infty$;
        \item  $\displaystyle \sum_{\substack{{n = -\infty}\\{n \neq 0}}}^\infty a_n (\xi_n - \alpha)^{k-\ell} = 0 \,\, \text{ for } \,\, 1 \leq \ell \leq k$.
    \end{enumerate}
    Then
    \begin{equation}\label{EPSeq}
        (\mathbb{EP})(E,k,\alpha) = \inf_{\{a_n\} \in X \setminus \{0\} }  \frac{\textstyle \sum_{n \in \Z^*} c_n a_n^2 (\xi_n - \alpha)^{2k}}{\textstyle \sum_{n \in \Z^*} c_n a_n^2 }.
    \end{equation}
\end{lemma}

\begin{proof}
    We work with real-valued sequences because an extremizer for $(\mathbb{EP})(E,k,\alpha)$ can be taken real entire. Indeed, we know by Proposition \ref{QualP} that an $f \in \X_k$ attaining $(\mathbb{EP})(E,k,\alpha)$ exists, but we can further decompose $f(z) = g(z) - i h(z)$, where $g, h \in \X_k$ are real entire functions, by setting $g = (f + f^*)/2$ and $h = i(f - f^*)/2$ . One then has
    \begin{equation*}
        \frac{\|(z - \alpha)^k f\|_{\H (E)}^2}{\|f\|_{\H (E)}^2} = \frac{\|(z - \alpha)^k g\|_{\H (E)}^2 + \|(z - \alpha)^k h\|_{\H (E)}^2}{\|g\|_{\H (E)}^2 + \|h\|_{\H (E)}^2},
    \end{equation*}
    which, from elementary considerations, implies $g$ and $h$ must also be extremizers.
    
    Given a non-identically zero $f \in \X_k$ that is real entire, let us set $a_n := f(\xi_n)/A'(\xi_n)$. By (C1) and (C3), we may use the interpolation formula and Parseval's identity in $\H(E)$ given respectively by (\ref{Int}) and (\ref{Pars}). In particular, by (\ref{Pars}) we have
    \begin{equation*}
        \| f\|_{\H(E)}^2 =  \pi \sum_{\substack{{n = -\infty}\\{n \neq 0}}}^\infty c_n a_n^2 \quad \text{and } \quad \| (z - \alpha)^k f\|_{\H(E)}^2 =  \pi \sum_{\substack{{n = -\infty}\\{n \neq 0}}}^\infty c_n a_n^2 (\xi_n - \alpha)^{2k}.
    \end{equation*}
    The fact that the norm $\| (z - \alpha)^k f\|_{\H(E)}^2$ is finite implies the sequence $\{a_n\}$ satisfies condition (i).
    And indeed we have
    \begin{equation}\label{EqSeqNorm}
        \frac{\| (z - \alpha)^k f\|_{\H(E)}^2}{\| f\|_{\H(E)}^2} =\frac{\textstyle \sum_{n \in \Z^*} c_n a_n^2 (\xi_n - \alpha)^{2k}}{\textstyle \sum_{n \in \Z^*} c_n a_n^2 }.
    \end{equation}
    Further, we have that $f \in \X_k$ if and only if $g_\ell (z) := (z-\alpha)^{k - \ell + 1} f(z) \in \H(E)$ for all $1 \leq \ell \leq k$, which gives the  constraints
    \begin{equation*}
        g_1(\alpha)= \ldots = g_{k}(\alpha)=0,
    \end{equation*}
    and this implies item (ii). In fact,
    if $\alpha$ is not a zero of $A$, then by (\ref{Int})
    \begin{equation*}
        0 = g_\ell (\alpha) = \sum_{\substack{{n = -\infty}\\{n \neq 0}}}^\infty \frac{ g_\ell (\xi_n)}{K(\xi_n, \xi_n)} K(\xi_n,\alpha) = \sum_{\substack{{n = -\infty}\\{n \neq 0}}}^\infty \frac{ g_\ell (\xi_n)}{A'(\xi_n)} \frac{A(\alpha)}{\alpha - \xi_n}= -\sum_{\substack{{n = -\infty}\\{n \neq 0}}}^\infty \frac{f(\xi_n)}{A'(\xi_n)} (\xi_n - \alpha)^{k - \ell }A(\alpha),
    \end{equation*}
    and since $A(\alpha) \neq 0$, this is equivalent to 
    \begin{equation}\label{SeqConst}
        \sum_{\substack{{n = -\infty}\\{n \neq 0}}}^\infty a_n (\xi_n - \alpha)^{k - \ell} = 0
    \end{equation}
   for all $1 \leq \ell \leq k$. If $\alpha$ is a zero of $A$, then $A'(\alpha) \neq 0$ since by condition (C1)  the zeros of $A$ are simple. Moreover, for any $g \in \H(E)$ its representation (\ref{Int}) converges uniformly in compact sets, which allows us to differentiate inside the integral to obtain
    \begin{align*}
        g'(z) = \sum_{\substack{{n = -\infty}\\{n \neq 0}}}^\infty \frac{g(\xi_n)}{K(\xi_n, \xi_n)}\frac{\d}{\d z}K(\xi_n, z) =\sum_{\substack{{n = -\infty}\\{n \neq 0}}}^\infty \frac{g(\xi_n)}{A'(\xi_n)}\bigg\{  \frac{A'(z)(z - \xi_{n}) - A(z)}{(z - \xi_n)^2}   \bigg\}.
    \end{align*}    
    Applying the above formula to
    \begin{equation*}
        g'_\ell(z) = g_{\ell + 1}(z) + (z-\alpha)g'_{\ell + 1}(z)
    \end{equation*}
   evaluated at $\alpha$, we obtain
   \begin{align*}
       g_{\ell + 1}(\alpha) = g'_{\ell} (\alpha) = -\frac{g_\ell(\alpha)}{A'(\alpha)}\frac{A''(\alpha)}{2} + \sum_{\substack{{n = -\infty}\\{n \neq 0}\\{\xi_n \neq \alpha}}}^\infty \frac{g_\ell(\xi_n)}{A'(\xi_n)} \frac{A'(\alpha)}{(\alpha - \xi_n)} = -\sum_{\substack{{n = -\infty}\\{n \neq 0}\\{\xi_n \neq \alpha}}}^\infty \frac{f(\xi_n)}{A'(\xi_n)}(\xi_n - \alpha)^{k - \ell}A'(\alpha)
   \end{align*}
    for $1 \leq \ell \leq k$, which is again equivalent to (\ref{SeqConst}), establishing that $\{a_n \} \in X$.

    Conversely, if one has a non-zero sequence $\{a_n\} \in X$, the summability condition (i) allows us to construct  $g \in \H(E)$ using (\ref{Int}) by setting $g(\xi_n) = (\xi_n - \alpha)^k a_n A'(\xi_n)$. Condition (ii) in turn guarantees that $g$ has a zero of order at least $k$ at $\alpha$. So taking $f(z) = (z - \alpha)^{-k} g(z)$, this entire function will occur in $\X_k$, and (\ref{EqSeqNorm}) holds, concluding the lemma.
\end{proof}

\noindent {\sc Remark.} In establishing the above equivalence, we use the interpolation formula and Parseval's identity with respect to the zeros of the $A$ function. De Branges's result \cite[Theorem 22]{dB} actually gives these formulas with respect to the zeros of \emph{any} of the functions $e^{i \beta} E(z) - e^{- i\beta}E^*(z)$ when (C1) holds and $0 \leq \beta < \pi$ is such that $e^{i \beta} E(z) - e^{- i\beta}E^*(z) \notin \H(E)$. Of course, we specialize to the case $\beta = \pi/2$ here, but an analogous version of the above lemma may be obtained for other values of $\beta$ when the condition corresponding to (C3) holds. Of particular note is the case $\beta = 0$, which implies we may do the same things for the zeros of the $B$ function when $B \notin \H(E)$.
    
\begin{proof}[Proof of Theorem \ref{EPThm}]
    Since the hypotheses also apply here, we can invoke Lemma \ref{SeqEquiv}. We essentially solve the right-hand side of (\ref{EPSeq}), except in a particular case we will single out below. The coefficients $\{c_n\}$ and the space $X$ are the same as before.

    Now let $I := \{i_1 < i_2 < \ldots < i_k\}$ be a set of indices such that $\xi_{i_1} < \xi_{i_2}< \ldots < \xi_{i_k}$ are a choice of the $k$-th closest zeros to $\alpha$. That is, for all $i \in I$ and $n \notin I$, $|\xi_n - \alpha| \geq |\xi_{i} - \alpha|$. If $\{a_n\} \in X$, given the constraints (\ref{SeqConst}), we may write the following system of $k$ equations
    \begin{eqnarray*}
        &a_{i_1} (\xi_{i_1} - \alpha)^{k-1} + \ldots + a_{i_k} (\xi_{i_k} - \alpha)^{k-1}= - \sum_{n \notin I} a_n (\xi_n - \alpha)^{k-1},\\
        &a_{i_1} (\xi_{i_1} - \alpha)^{k-2} + \ldots + a_{i_k} (\xi_{i_k} - \alpha)^{k-2} = - \sum_{n \notin I} a_n (\xi_n - \alpha)^{k-2},\\
        &\ldots\\
        &a_{i_1}  + \ldots + a_{i_k} = - \sum_{n \notin I} a_n ,
    \end{eqnarray*}
    which in matrix notation becomes
    \begin{equation*}\label{MatrixEq}
        (a_{i_1} , \ldots , a_{i_k})\,\mc T = (-S_1, \ldots, -S_k),
    \end{equation*}
    where $\mc T$ is the matrix given by $\mc T_{mj} = (\xi_{i_m} - \alpha)^{k-j}$ for $1\leq m,j \leq k$ and $S_j = \sum_{n \notin I} a_n (\xi_n - \alpha)^{k-j}$ for $1\leq j \leq k$ \footnote{Here we adopt the convention $0^0 = 1$ whenever it may occur.}. In particular,
    \begin{equation}\label{RelSa}
        a_{i_j} = - \sum_{r=1}^k S_r (\mc T^{-1})_{rj} \  \ \text{ for } \ \ 1 \leq j \leq k.
    \end{equation}

    If $k \geq 2$, notice that for one $i^* \in I$, we have the strict inequality $|\xi_{i^*} - \alpha|<|\xi_{\ell} - \alpha|$ for all $\ell \notin I$, since there are at most two $\xi_i$ such that $|\xi_i - \alpha| = \min_{A(\xi) =0} |\xi - \alpha|$. If $k = 1$, suppose for now this minimum is attained by exactly one of $\xi_{i_1}$ and $\xi_{i_2}$ , ($\alpha$ can be the midpoint of two zeros of $A$, but we deal with this case later,) which we call $\xi_{i^*}$ so that we also have $|\xi_{i^*} - \alpha|<|\xi_{\ell} - \alpha|$ for all $\ell \notin I$. 
    
    Now for each $\ell \notin I$ construct the following sequence in $X$:
   \begin{equation*}
        a_n =\begin{cases}
        1, & {\rm if} \ \ n = \ell;\\
        0\,, &{\rm if} \ \ n \notin I \cup \{\ell\};\\
         - \sum_{r=1}^k S_r (\mc T^{-1})_{rj},&{\rm if} \ \ n=i_j \in I.
        \end{cases}
   \end{equation*}
   implying that, by the remark in the previous paragraph and the definition of $I$,
   \begin{align}\label{StrictIneq}
       \lambda_0^{2k} := (\mathbb{EP})(E,k,\alpha) \leq \frac{\textstyle \sum_{n \in I \cup \{\ell\}} c_n a_n^2 (\xi_n - \alpha)^{2k}}{\textstyle \sum_{n \in I \cup \{\ell\}} c_n a_n^2 } < \frac{\textstyle \sum_{n \in I \cup \{\ell\}} c_n a_n^2 (\xi_\ell - \alpha)^{2k}}{\textstyle \sum_{n \in I \cup \{\ell\}} c_n a_n^2 } = (\xi_\ell - \alpha)^{2k}.
   \end{align}

   Now define the functionals $F$ and $G$ on the space of sequences $\{ a_n\}_{n \notin I}$ satisfying $ \sum_{n \notin I} c_n a_n^2 \xi_n^{2k} < \infty$ by
   \begin{equation*}
       F(\{ a_n\}) = \sum_{j = 1}^{k}c_{i_j} \bigg( \sum_{r=1}^k S_r (\mc T^{-1})_{rj}\bigg)^2(\xi_{i_j} - \alpha)^{2k} + \sum_{n \notin I} c_n a_n^2 (\xi_n - \alpha)^{2k} 
   \end{equation*}
   and
   \begin{equation*}
       G(\{ a_n\})=  \sum_{j = 1}^{k}c_{i_j} \bigg( \sum_{r=1}^k S_r (\mc T^{-1})_{rj}\bigg)^2 + \sum_{n \notin I} c_n a_n^2 ,
   \end{equation*}
   where again $S_j = \sum_{n \notin I} a_n (\xi_n - \alpha)^{k-j}$ for $1\leq j \leq k$.

   From Proposition \ref{QualP} and Lemma \ref{SeqEquiv}, we know there is $a^* = \{a_n^*\}_{n \notin I}$ extremizer to (\ref{EPSeq}). If we perturb this sequence and consider for small $\varepsilon > 0$ and for $m \notin I $ the function 
   \begin{equation*}
       \varphi_m(\eps) = \frac{ F(a^* + \eps e^m)}{G(a^* + \eps e^m)},
   \end{equation*}
   where $e^m$ is the sequence with entries 
   \begin{equation*}
       e^m_n = \begin{cases}
           1 \text{ if } m = n;\\
           0  \text{   otherwise},
       \end{cases}
   \end{equation*}
   then $\varphi_m$ is differentiable and from extremality we have $\varphi_m'(0) = 0$, so the sequence must satisfy the Lagrange multiplier equations
    \begin{equation*}
       \frac{\partial F}{\partial a_m} \left(a^* \right) = \lambda_0^{2k} \frac{\partial G}{\partial a_m} \left (a^* \right ) \, \text{for } m \notin I,
    \end{equation*}
    or, equivalenty, defining $(a^*_{i_1} , \ldots , a^*_{i_k})$ according to (\ref{RelSa}),
    \begin{equation}\label{SeqTermsPrev}
        c_m a_m^* \left((\xi_m - \alpha)^{2k} -\lambda_0^{2k}\right) = \sum_{j = 1}^{k}c_{i_j} a_{i_j}^* \bigg( \sum_{r=1}^k (\xi_m - \alpha)^{k-r} (\mc T^{-1})_{rj}\bigg)\left((\xi_{i_j}- \alpha)^{2k} -\lambda_0^{2k}\right) \, \text{for } m \notin I.
    \end{equation}
    By (\ref{StrictIneq}), we can divide both sides by $c_m \left((\xi_m - \alpha)^{2k} -\lambda_0^{2k}\right)$ to say
    \begin{equation}\label{SeqTerms}
        a_m^*= \frac{1}{c_m \left((\xi_m - \alpha)^{2k} -\lambda_0^{2k}\right)} \sum_{j = 1}^{k}c_{i_j} a_{i_j}^* \bigg( \sum_{r=1}^k (\xi_m - \alpha)^{k-r} (\mc T^{-1})_{rj}\bigg)\left((\xi_{i_j}- \alpha)^{2k} -\lambda_0^{2k}\right) \, \text{for } m \notin I,
    \end{equation}
    multiply by $(\xi_m - \alpha)^{k - \ell}$, with $1 \leq \ell \leq k$, and sum over $m \notin I$ to get
    \begin{equation*}
        S^*_\ell =\sum_{m \notin I} a_m^* (\xi_m - \alpha)^{k - \ell} = \sum_{j = 1}^{k}a^*_{i_j} \sum_{m \notin I} \frac{c_{i_j} \bigg( \sum_{r=1}^k (\xi - \alpha)_m^{k-r} (\mc T^{-1})_{rj}\bigg)\left((\xi_{i_j}- \alpha)^{2k} -\lambda_0^{2k}\right)}{c_m \left((\xi_m - \alpha)^{2k} -\lambda_0^{2k}\right)}(\xi_m - \alpha)^{k - \ell}.
    \end{equation*}

    Recalling that (4.\ref{MatrixEq}) tells us
    \begin{equation*}
        S^*_\ell = - \sum_{j = 1}^k a_{i_j}^* \mc T_{j \ell} = -  \sum_{j = 1}^k a_{i_j}^* (\xi_{i_j} - \alpha)^{k - \ell},
    \end{equation*}
    we can subtract the two expressions for $S^*_\ell$ to obtain the equality
    \begin{equation}\label{AllConstrFormula}
        \sum_{j = 1}^{k}a^*_{i_j} \left [ \sum_{m \notin I} \frac{c_{i_j} \bigg( \sum_{r=1}^k (\xi_m - \alpha)^{k-r} (\mc T^{-1})_{rj}\bigg)\left((\xi_{i_j}- \alpha)^{2k} -\lambda_0^{2k}\right)}{c_m \left((\xi_m - \alpha)^{2k} -\lambda_0^{2k}\right)}(\xi_m - \alpha)^{k - \ell} + (\xi_{i_j} - \alpha)^{k - \ell} \right ] = 0
    \end{equation}
    for each $1 \leq \ell \leq k$. This can be equivalently stated in terms of matrix notation as saying that $(a^*_{i_1} ,\ldots , a^*_{i_k}) \in \mathrm{ker} \, \mc W (\lambda_0)$ where, for a generic complex $\lambda$, the $k \times k$ matrix $\mc W (\lambda)$ is given by
    \begin{align}
        \label{defW} \big(\mc W (\lambda)\big)_{j \ell} &= \sum_{m \notin I} \frac{c_{i_j} \bigg( \sum_{r=1}^k (\xi_m - \alpha)^{k-r} (\mc T^{-1})_{rj}\bigg)\left((\xi_{i_j}- \alpha)^{2k} -\lambda^{2k}\right)}{c_m \left((\xi_m - \alpha)^{2k} -\lambda^{2k}\right)}(\xi_m - \alpha)^{k - \ell} + (\xi_{i_j} - \alpha)^{k - \ell}  \\
        \nonumber&= c_{i_j} \left((\xi_{i_j}- \alpha)^{2k} -\lambda^{2k}\right) \sum_{\substack{{m = -\infty}\\{m\neq 0}}}^\infty \frac{ \bigg( \sum_{r=1}^k (\xi_m - \alpha)^{k-r} (\mc T^{-1})_{rj}\bigg)}{c_m \left((\xi_m - \alpha)^{2k} -\lambda^{2k}\right)}(\xi_m - \alpha)^{k - \ell} \\
        \nonumber&= c_{i_j} \left((\xi_{i_j}- \alpha)^{2k} -\lambda^{2k}\right) \left( (\mc T^{-1})^\top \Q (\lambda) \right)_{j\ell},
    \end{align}
    where we have used that
    \begin{equation*}
        \sum_{r = 1}^{k} (\xi_{i_n} - \alpha)^{k - r} (\mc T^{-1})_{rj}= \begin{cases}
            1 \text{ if } n = j,  \\
            0 \text{ if } n \neq j
        \end{cases}
    \end{equation*}
   for $1 \leq n \leq k$, and $\Q (\lambda)$ is the $k \times k$ matrix with entries
    \begin{equation}\label{QMatr}
        \big(\mc Q (\lambda)\big)_{n \ell} = \sum_{\substack{{m = -\infty}\\{m\neq 0}}}^\infty \frac{ (\xi_m - \alpha)^{2k - n - \ell}}{c_m \left((\xi_m - \alpha)^{2k} -\lambda^{2k}\right)}
    \end{equation}
    for $1 \leq n, \ell \leq k $.

    If instead $0 < \lambda^* < \min_{\ell \notin I} |\xi_\ell - \alpha|$ satisfies $\det \mc W(\lambda^*) = 0$, pick $(a^*_{i_1} , \ldots , a^*_{i_k}) \in \mathrm{ker} \, \mc W (\lambda^*) \setminus \{\mathbf{0}\}$, and using the formula given by (\ref{SeqTerms}) we may build a sequence $\{a_n^*\}_{\Z^*} \in X$. In fact, the constraints (\ref{SeqConst}) come from the expression (\ref{AllConstrFormula}) and the summability $\sum c_n (a_n^*)^2 \xi_n^{2k} < \infty$ follows from (\ref{SeqTerms}) and the fact that $K(0,z) = A(0)B(z)/z$ implies $\sum \frac{1}{c_n \xi_n^2} < \infty$. Undoing the step from (\ref{SeqTerms}) to (\ref{SeqTermsPrev}) and summing over $m \notin I$, we get
    \begin{align*}
        \sum_{m \notin I} c_m (a_m^*)^2 \big((\xi_m - \alpha)^{2k} -& (\lambda^*)^{2k})\big) = \sum_{m \notin I} a_m^* \sum_{j = 1}^{k}c_{i_j} a_{i_j}^* \bigg( \sum_{r=1}^k (\xi_m - \alpha)^{k-r} (\mc T^{-1})_{rj}\bigg)\big((\xi_{i_j}- \alpha)^{2k} - (\lambda^*)^{2k}\big) \\
        &=  \sum_{j = 1}^{k}c_{i_j} a_{i_j}^* \bigg( \sum_{r=1}^k S_r^* (\mc T^{-1})_{rj}\bigg)\big((\xi_{i_j}- \alpha)^{2k} - (\lambda^*)^{2k}\big) \\
        &= - \sum_{j = 1}^{k}c_{i_j} (a_{i_j}^*)^2 \big((\xi_{i_j}- \alpha)^{2k} - (\lambda^*)^{2k}\big),
    \end{align*}
    where in the last equality we have used (\ref{RelSa}). Upon rearranging the terms, the equation above becomes
    \begin{equation*}
        \frac{\sum_{n \in \Z^*} c_n a_n^2 (\xi_n - \alpha)^{2k}}{\textstyle \sum_{n \in \Z^*} c_n a_n^2 } =  (\lambda^*)^{2k},
    \end{equation*}
    whence the conclusion follows that the desired $\lambda_0$ is the smallest such $\lambda^*$.
    
    The job now mostly consists in describing $\lambda_0$ as the zero of the simpler function appearing in the theorem statement. To do this we look at $\Q(\lambda)$, since
    \begin{align*}
        \det \mc W (\lambda) = \left( \prod_{i \in I}c_{i} \left((\xi_{i}- \alpha)^{2k} -\lambda^{2k}\right) \right) \left( \det \mc T^{-1}\right) \det \Q(\lambda),
    \end{align*}
    which we can further simplify by using the partial fraction decomposition  $\frac{2k x^{2k - s -1} y^n}{x^{2k} - y^{2k}} = \sum_{r = 0}^{2k - 1}\frac{\omega^{-rs}}{x - \omega^r y}$, as in \cite[Lemma 26]{CMCP}, where $\omega := e^{i \pi /k}$ and $0 \leq s \leq 2k -1$. Recall also the formula for $C(z)$
    \begin{equation}\label{Tangent}
        C(z) = \sum_{m = 1}^\infty \frac{2z}{c_m (\xi_m^2 -z^2)},
    \end{equation}
    which converges uniformly in compact subsets of $\C$ away from the zeros of $A$. Let us now look at the partial sums of (\ref{QMatr}). Since by (C2) the zeros of $A$ are symmetric, i.e. $\xi_{-m} = - \xi_m$, we have
    \begin{align*}
        \sum_{0 < |m| \leq M} &\frac{ (\xi_m - \alpha)^{2k - n - \ell}}{c_m \left((\xi_m - \alpha)^{2k} -\lambda^{2k}\right)} = \frac{1}{2k \lambda^{\ell+n-1}} \sum_{0 < |m| \leq M} \frac{2k \lambda^{\ell+n-1}(\xi_m - \alpha)^{2k - n - \ell}}{c_m\left((\xi_m - \alpha)^{2k} -\lambda^{2k}\right)}\\
        &= \frac{1}{2k \lambda^{\ell+n-1}} \sum_{0 < |m| \leq M} \sum_{r = 0}^{2k - 1}\frac{\omega^{-r(n + \ell -1)}}{c_m(\xi_m - \alpha - \omega^r \lambda)} \\
        &= \frac{1}{2k \lambda^{\ell+n-1}} \sum_{r = 0}^{2k - 1}\sum_{m = 1}^M \frac{\omega^{-r(n + \ell -1)}}{c_m(\xi_m - \alpha - \omega^r \lambda)} + \frac{\omega^{-r(n + \ell -1)}}{c_m(-\xi_m - \alpha - \omega^r \lambda)} \\
        &=  \frac{1}{2k \lambda^{\ell+n-1}} \sum_{r = 0}^{2k - 1}\omega^{-r(n + \ell -1)}\sum_{m = 1}^M \frac{2(\alpha + \omega^r \lambda)}{c_m(\xi_m^2 - (\alpha + \omega^r \lambda)^2)},
    \end{align*}
    which by formula (\ref{Tangent}) converges to 
    \begin{equation*}
        \frac{1}{2k \lambda^{\ell+n-1}} \sum_{r = 0}^{2k - 1}\omega^{-r(n + \ell -1)}C(\alpha + \omega^r \lambda) = \frac{\big(\V (\lambda)\big)_{n \ell}}{2k \lambda^{\ell+n-1}} 
    \end{equation*}
    as $M \to \infty$. Thus, we have $\det \mc W(\lambda) = 0$ if and only if
    \begin{equation*}
         \left( \prod_{i \in I}c_{i} \left((\xi_{i}- \alpha)^{2k} -\lambda^{2k}\right) \right) \det \V(\lambda)  = 0,
    \end{equation*}
    since the exponent of the factors $\lambda^{- (\ell + n - 1)}$ in each entry becomes constant when the determinant is expanded. Observe that from the expression (\ref{defW}) and the inequality (\ref{StrictIneq}) we can see that the function $\lambda \mapsto \det \mc W (\lambda)$ is meromorphic in $\C$ and continuous in $(0, \min_{\ell \notin I} |\xi_\ell - \alpha|)$, and therefore so is $\lambda \mapsto \left( \prod_{i \in I}c_{i} \left((\xi_{i}- \alpha)^{2k} -\lambda^{2k}\right) \right) \det \V(\lambda)$.  Now $\lambda$ is a real solution of the above if and only if
    \begin{equation}\label{AlmostFinal}
        \left( \prod_{i \in I} \left(\xi_{i}- \alpha -\lambda \right)\left(\xi_{i}- \alpha + \lambda \right)  \right)  \det \V(\lambda) = 0.
    \end{equation}

    Recall that inequality (\ref{StrictIneq}) and item (i) of Proposition \ref{QualP} also established that $\lambda_0$ is within the interval $\left(0, \min_{\ell \notin I}|\xi_{\ell} - \alpha|\right)$. We know the zeros of $A(\alpha + x)$ in this interval are precisely the $\xi_i - \alpha$ and those of $A(\alpha - x)$ precisely the $- \xi_i +\alpha$. So we can replace the product in (\ref{AlmostFinal}) by these respective functions obtaining $\lambda_0$ must be the first positive solution of
    \begin{equation*}
        A(\alpha + \lambda)A(\alpha - \lambda) \det \V(\lambda) = 0.
    \end{equation*}

    Now we come back to the assumption we made when $k = 1$. We supposed $\alpha$ is not the midpoint of two consecutive zeros of $A$, but we will now show that even in this case the first positive zero of (\ref{RealSol}) yields our answer. In fact, if there exists $i \in \Z^*$ such that $\alpha = \frac{1}{2}(\xi_i + \xi_{i+1}) $, then by (\ref{Int}), we have that
    \begin{align*}
        \|f\|_{\H(E)}^2 &= \sum_{A(\xi) = 0} \frac{|f(\xi)|^2}{K(\xi, \xi)} = \sum_{A(\xi) = 0} \frac{|\xi - \alpha|^2}{|\xi - \alpha|^2}\frac{|f(\xi)|^2}{K(\xi, \xi)} \\
        &\leq \frac{1}{|\xi_i - \alpha|^2}\sum_{A(\xi) = 0} |\xi - \alpha|^2\frac{|f(\xi)|^2}{K(\xi, \xi)} = \frac{1}{|\xi_i - \alpha|^2} \|(z-\alpha) f\|_{\H(E)}^2,
    \end{align*}
    so $\lambda_0 \geq |\xi_i - \alpha| $. Now since $|\xi_i - \alpha| = |\xi_{i+1} - \alpha|$, equality above can be obtained only if $f(\xi) = 0$ when $\xi$ is a zero of $A$ neither equal to $\xi_i$ nor to $\xi_{i+1}$. Hence extremizers can only be of the form
    \begin{align*}
        f(\xi_{i})\frac{K(\xi_{i},x)}{K(\xi_{i},\xi_{i})} + f(\xi_{i+1})\frac{K(\xi_{i+1},x)}{K(\xi_{i+1},\xi_{i+1})} =\bigg\{\frac{f(\xi_i)}{A'(\xi_i)} \frac{1}{z - \xi_i} + \frac{f(\xi_{i+1})}{A'(\xi_{i+1})} \frac{1}{z - \xi_{i+1}}\bigg\} A(x),
    \end{align*}
    which can possibly occur in $\X_k$ if we impose the symmetry condition $ f(\xi_i)/A'(\xi_i) = - f(\xi_{i+1})/A'(\xi_{i+1})$. For this reason, we actually have the equality $\lambda_0 = |\xi_i - \alpha| = \frac{1}{2}( \xi_{i+1} - \xi_i )$. 

    We only have to show that this is the first zero of
    \begin{align*}
         A(\alpha + \lambda) A(\alpha - \lambda) \left( C(\alpha + \lambda) - C(\alpha - \lambda)\right) =  B(\alpha + \lambda) A(\alpha - \lambda) -  A(\alpha + \lambda) B(\alpha - \lambda).
    \end{align*}
    Plugging in $\frac{1}{2}( \xi_{i+1} - \xi_i )$ for $\lambda$ yields
    \begin{equation*}
        B(\xi_{i+1}) A(\xi_{i}) -  A(\xi_{i+1}) B(\xi_{i}) = 0,
    \end{equation*}
    but we need to rule out the existence of smaller zeros. To do this, recall that to every $E$ Hermite-Biehler  function satisfying condition (C1) we may associate an analytic phase function $\varphi:\mc W \to \C $, where $\R \subseteq \mc W \subseteq \C $ is an open simply connected set, satisfying
    \begin{equation}\label{Phase}
        E(x)e^{ i \varphi (x)} \in \R \ \ \text{ for all} \ \ x \in \R.
    \end{equation}
    It is immediate from (C1) and the fact that both $A$ and $B$ are real entire that $\gamma$ is a zero of $A$ if and only if $\varphi(\gamma) \equiv \frac{\pi}{2} (\bmod \pi)$. 
    
    Now, an expression equivalent to (\ref{Phase}) is
    \begin{equation*}
        e^{2 i \varphi (x)}  = \frac{E^*(x)}{E(x)} = \frac{A(x) + iB(x)}{A(x) - iB(x)} 
    \end{equation*}
    and from this it follows that $B(\beta)A(\gamma) - A(\beta)B(\gamma) = 0$ if and only if $\varphi(\beta) \equiv \varphi(\gamma) \bmod \pi$. 
    
    Suppose by contradiction there is a $0 < \lambda^* < \frac{1}{2}(\xi_{i+1} - \xi_{i})$ such that $B(\alpha + \lambda^*) A(\alpha - \lambda^*) -  A(\alpha + \lambda^*) B(\alpha - \lambda^*) = 0$. Thus, we would need $\varphi(\alpha + \lambda^*) \equiv \varphi(\alpha - \lambda^*) \bmod \pi$ . Since $\varphi(x)$ is a strictly increasing function of $x \in \R$ \cite[Problem 48]{dB} it follows that there exists a $\gamma \in [\alpha - \lambda^*, \alpha + \lambda^* ]$ such that $\varphi(\gamma) \equiv \frac{\pi}{2} \bmod \pi$, i.e., $A(\gamma) = 0$. But $\xi_{i}< \alpha - \lambda^*< \alpha + \lambda^* < \xi_{i+1}$, while $\xi_{i} \text{ and } \xi_{i+1}$ are consecutive zeros of $A$, which is a contradiction. Hence $\lambda_0 =  \frac{1}{2}( \xi_{i+1} - \xi_i )$ is indeed the first positive zero, concluding the proof. 
\end{proof}

\noindent {\sc Remark.} In the particular case $k = 1$ and $\alpha$ is a midpoint of two consecutive zeros $\xi_i < \xi_{i + 1}$, we saw in the proof above that extremizers must be complex multiples of the function
\begin{equation*}
    f(z) = \frac{A(z)}{(z - \xi_i)(z - \xi_{i+1})},
\end{equation*}
but we may also classify extremizers in the other cases contemplated in this theorem.  In view of the proof above, we can obtain a real extremizing sequence $\{a_n^*\} \in X$ to the right-hand side of (\ref{EPSeq}) by taking $(a_{i_1}^*, \ldots, a_{i_k}^*) \in \ker \mc W (\lambda_0) \setminus \{\pmb 0\}$ and building the rest of the sequence by (\ref{SeqTerms}). Applying formula (\ref{Int}), extremizers must then be in the span over $\C$ of the entire functions
\begin{equation*}
    f(z) = \sum_{\substack{{n = -\infty}\\ {n \neq 0}}}^\infty a_n^* \frac{A(z)}{z - \xi_n}
\end{equation*}
that come from sequences $\{a_n^*\} \in X$ constructed in this way, as the initial reduction in the proof of Lemma \ref{SeqEquiv} guarantees.

\subsection{Proof of Theorem \ref{EPSolution}} A few steps are missing in order to conclude Theorem \ref{EPSolution}, which instead follows directly from a slight adaptation of the argument. In fact, as Carneiro, Chirre, and Milinovich establish in \cite[Section 5.2]{CChiM}, the spaces $(\H_{G, \pi \Delta}, \langle \cdot ,\cdot\rangle_{\H_{G, \pi \Delta}})$ are isometrically isomorphic to certain de Branges spaces. They give a procedure that allows one to obtain a (not necessarily unique) Hermite-Biehler function $E$  from the expression of the reproducing kernels of $\H_{G, \pi \Delta}$. It consists of defining
\begin{equation*}
    L(w, z) := 2 \pi i (\overline{w} - z) K_{G, \pi \Delta}(w,z), 
\end{equation*}
and then considering the entire function
\begin{equation*}
    E(z) := L(i,i)^{-\frac{1}{2}} L(i, z),
\end{equation*}
which is Hermite-Biehler (see \cite[Problems 50-53]{dB} and \cite[Appendix A]{CChaLM} for details). By \cite[Theorem 23]{dB}, $E$ is such that $\H_{G,\pi \Delta} = \H(E)$ as sets and
\begin{align*}
    \| F \|_{\H_{G, \pi \Delta}} = \left( \int_\R |F(x)|^2 \,  W_G(x) \, \d x \right)^{\frac{1}{2}} =   \left(\int_{-\infty}^\infty \left| \frac{F(x )}{E(x)}\right|^2 \d x  \right)^{\frac{1}{2}}= \|F\|_{\H (E)} 
\end{align*}
for each $F$ therein. 
Such an $E$ can be seen to satisfy conditions (C1) and (C2), so that Theorem \ref{EPSolution} is thus a consequence of the following.

\begin{corollary}\label{CorTieBack}
    Let $E$ be a Hermite-Biehler function satisfying (C1) and (C2) and $\X_1 \neq \{ 0\}$. Setting $\lambda_0 = \big((\mathbb{EP})(E,1,\alpha)\big)^{1/2}$, we have that $\lambda_0$ is the first positive zero of the equation
    \begin{equation*}
        K(\alpha + \lambda, \alpha - \lambda) = 0.
    \end{equation*}
\end{corollary}

\begin{proof}
    If $A \notin \H (E)$ then we can apply Theorem \ref{EPThm} when $k = 1$, which states $\lambda_0$ is the first zero of 
    \begin{align*}
         A(\alpha + \lambda) A(\alpha - \lambda) \left( C(\alpha + \lambda) - C(\alpha - \lambda)\right) =  B(\alpha + \lambda) A(\alpha - \lambda) -  A(\alpha + \lambda) B(\alpha - \lambda),
    \end{align*}
  which by (\ref{RepKernel}) is 
  \begin{equation*}
      2 \pi \lambda \, K(\alpha + \lambda, \alpha - \lambda).
  \end{equation*}
  If $B \notin \H (E)$, we may use the interpolation formula and Parseval's identity corresponding to the zeros of $B$. That is, as in (\ref{Int}) and (\ref{Pars}), we have by \cite[Theorem 22]{dB} that 
  \begin{equation*}
      f(z) = \sum_{B(\eta) = 0} \frac{f(\eta)}{K(\eta, \eta)}K(\eta, z) \ \ \text{   and   } \ \ \|f\|_{\H(E)}^2 = \sum_{B(\eta) = 0} \frac{|f(\eta)|^2}{K(\eta, \eta)},
  \end{equation*}
   and an analogous version of Lemma \ref{SeqEquiv} holds (see also the remark immediately following that lemma). By (C2), the zeros of $B$ are also symmetric in the sense that $B(\eta) = 0 $ if and only if $B(-\eta) = 0$. Its non-negative zeros can be enumerated as $0 = \eta_0 < \eta_1 < \ldots$, and so defining $d_m = B'(\eta_m)/A(\eta_m)$, we likewise have the formula
   \begin{equation*}
       \frac{A(z)}{B(z)} = \frac{1}{d_0 z} + \sum_{m=1}^\infty \frac{2 z}{d_m (z^2 - \eta_m^2)}
   \end{equation*}
   converging uniformly in compact sets away from the zeros of $B$. The same procedure in the proof of Theorem \ref{EPThm} can now be repeated, in which case $\lambda_0$ may be seen to be the first positive zero of 
  \begin{equation*}
      B(\alpha + \lambda)B(\alpha - \lambda) \left( \frac{A(\alpha + \lambda)}{B(\alpha + \lambda)} -  \frac{A(\alpha - \lambda)}{B(\alpha - \lambda)}\right) = - 2 \pi \lambda \, K(\alpha + \lambda, \alpha - \lambda).
  \end{equation*}
    But for any given $E$, at most one of $A$ and $B$ can be in the space $\H (E)$, since by integrability considerations $E(z) = A(z) - i B(z)$ cannot itself occur in $\H (E)$, and the result holds.
\end{proof}

\noindent {\sc Remark.} The condition $\X_1 \neq \{ 0 \}$ is, of course, satisfied in the cases that interest us. For instance 
\begin{equation*}
    f(z) = \frac{\cos^2 \pi \Delta z /2}{1 - 4 \Delta^2 z^2}
\end{equation*}
is such that $z f(z) \in \H_{\pi \Delta}$. In general, we may conclude $\X_k$ is non-trivial if one of the associated functions $A(z)$ or $B(z)$ has at least $k+1$ zeros. If, say, $\xi_1, \ldots, \xi_{k+1}$ are such zeros, then one may obtain $g(z) = \frac{A(z)}{(z - \xi_1) \ldots (z - \xi_{k+1})} \in \H(E)$ (or the analogous for $B$), from an appropriate linear combination of the reproducing kernels $K(\xi_i, z)$. Since $|g(x)| \sim |A(x)||x|^{-k-1}$ for $|x| \gg 1$, this function satisfies $z^k g(z) \in \H(E)$.

\smallskip

\section{Concluding Remarks}

We conclude with an interesting qualitative property of the extremal problem in de Branges spaces.
\begin{proposition}
    Let $E$ be Hermite-Biehler and $k \in \N$. If $\X_k \neq \{0\}$, then the function $\alpha \in \R \mapsto (\mathbb{EP})(E, k, \alpha)$ is continuous.
\end{proposition}
\begin{proof}    
    By Proposition \ref{QualP}, for all $\alpha$ there exists an $f_\alpha \in \H (E)$ that attains the corresponding infimum in (\ref{EP}). That is, if we denote by 
    \begin{equation*}
        \Phi_{\alpha} (f):= \frac{\|(z-\alpha)^k f\|_{\H (E)}^2}{\|f\|_{\H (E)}^2}
    \end{equation*}
    for $0 \neq f \in \X_k$, we have $\Phi_{\alpha}(f_\alpha) = (\mathbb{EP})(E, k, \alpha)$. 
    
    Now fix $\alpha$ and let $\{\alpha_n\} \subseteq \R$ be a sequence converging to $\alpha$. We can invoke Proposition $\ref{QualP}$ again to obtain a sequence of extremizers $\{f_{\alpha_n}\} \subseteq \H(E)$ each satisfying $\Phi_{\alpha_n}(f_{\alpha_n}) = (\mathbb{EP})(E, k, \alpha_n)$. Without loss of generality, assume $\|f_\alpha\|_{\H (E)} = 1$ and $\|f_{\alpha_n}\|_{\H (E)} = 1$.

    On the one hand we have
    \begin{align}
        \nonumber (\mathbb{EP})(E, k, \alpha) &= \Phi_{\alpha}(f_\alpha) =  \int_{-\infty}^\infty \left| \frac{(x - \alpha)^k f_{\alpha}(x)}{E(x)}\right|^2 \d x \\
         \nonumber &= \lim_{n \rightarrow \infty} \int_{-\infty}^\infty \left| \frac{(x - \alpha_n)^k f_{\alpha}(x)}{E(x)}\right|^2 \d x = \lim_{n \rightarrow \infty} \Phi_{\alpha_n}(f_{\alpha}) \\
        \label{usc} &\geq \limsup_{n \rightarrow \infty} \,(\mathbb{EP})(E, k, \alpha_n) 
    \end{align}
    by dominated convergence, so that $\alpha \mapsto (\mathbb{EP})(E, k, \alpha)$ is upper semicontinuous.

    On the other hand, we can also see that
    \begin{align*}
        (\mathbb{EP})(E, k, \alpha) &\leq \Phi_{\alpha}(f_{\alpha_n}) =  \Phi_{\alpha_n}(f_{\alpha_n}) + \int_{-\infty}^\infty \left\{(x - \alpha)^{2k} - (x - \alpha_n)^{2k} \right\} \left| \frac{ f_{\alpha_n}(x)}{E(x)}\right|^2 \d x   \\
        &= (\mathbb{EP})(E, k, \alpha_n) + O\big(|\alpha - \alpha_n|\big),
    \end{align*}
    because
    \begin{align*}
        \bigg| \int_{-\infty}^\infty \big\{(x - \alpha)^{2k} - &(x - \alpha_n)^{2k} \big\}  \left| \frac{ f_{\alpha_n}(x)}{E(x)}\right|^2 \d x \bigg| \\
        &\leq \int_{-\infty}^\infty \left|(x - \alpha)^{2k} - (x - \alpha_n)^{2k} \right| \left| \frac{ f_{\alpha_n}(x)}{E(x)}\right|^2 \d x \\
        &\ll |\alpha -\alpha_n| \int_{-\infty}^\infty (\max \{|\alpha|,|\alpha_n|\}^{2k -1} + |x|^{2k - 1}) \left| \frac{ f_{\alpha_n}(x)}{E(x)}\right|^2 \d x \\
        &\ll |\alpha -\alpha_n|(\|f_{\alpha_n}\|_{\H (E)}^2 + \|z^k f_{\alpha_n}\|_{\H (E)}^2)\\
        &=  |\alpha -\alpha_n|( 1 +  (\mathbb{EP})(E, k, \alpha_n))
    \end{align*}
    and $(\mathbb{EP})(E, k, \alpha_n)$ is bounded in $n$ by (\ref{usc}), so that 
    \begin{equation*}
        (\mathbb{EP})(E, k, \alpha) \leq \liminf_{n \rightarrow \infty} \, (\mathbb{EP})(E, k, \alpha_n),
    \end{equation*}
    i.e., it is  lower semicontinuous. Hence, it is also continuous.
\end{proof}

\section{Acknowledgements}
I am grateful to Emanuel Carneiro for the constant guidance and support throughout the writing of this paper. I would also like to thank Tolibjon Ismoilov for the helpful discussions and Ariel S. Boiardi for his aid in the elaboration of the graphs of Figure 1.

\end{document}